\documentclass{amsart}

\usepackage{amsmath,amssymb,amsthm}
\usepackage[alphabetic]{amsrefs}
\usepackage{amsfonts}

\usepackage{color, float}  
\usepackage{comment}
\usepackage{algpseudocode}
\usepackage{algorithm}
\usepackage{caption}
\usepackage{subcaption}

\usepackage[all,cmtip]{xy}
\usepackage{tikz}
\usepackage{tikz-cd}
\usepackage{tikz}
\usepackage{caption}
\usepackage{tikz-cd}
\usetikzlibrary{calc,decorations.pathreplacing}
\usepackage{tikz}
\usetikzlibrary{calc}
\usetikzlibrary{arrows}
\usetikzlibrary{decorations.pathreplacing}
\usetikzlibrary{intersections}

\usetikzlibrary{matrix}
\usetikzlibrary{positioning}
\usetikzlibrary{arrows,calc, intersections, decorations.pathmorphing,backgrounds,positioning,fit, through}
\colorlet{LightBlue}{blue!50}

\usepackage{setspace}
\usepackage{cleveref}

\newtheorem{theorem}{Theorem}[section]
\newtheorem{proposition}[theorem]{Proposition}
\newtheorem{corollary}[theorem]{Corollary}
\newtheorem{porism}[theorem]{Porism}
\newtheorem{lemma}[theorem]{Lemma}
\newtheorem{observation}[theorem]{Observation}

\newtheorem{fact}[theorem]{Fact}
\newtheorem{notation}[theorem]{Notation}

\newtheorem{introthm}{Theorem}
\newtheorem{introcor}[introthm]{Corollary}

\theoremstyle{definition}
\newtheorem{definition}[theorem]{Definition}
\newtheorem{example}[theorem]{Example}

\theoremstyle{remark}
\newtheorem{remark}[theorem]{Remark}

\numberwithin{equation}{section}

\newcommand{\from}{\colon}

\newcommand{\F}{{\mathbb F}}
\newcommand{\Lam}{\mathcal{L}}
\newcommand{\BP}{\mathfrak{v}}

\DeclareMathOperator{\Aut}{Aut}
\DeclareMathOperator{\Inn}{Inn}
\DeclareMathOperator{\out}{Out}
\newcommand{\Out}{\out(\F)}

\DeclareMathOperator{\rk}{\mathfrak{r}}

\DeclareMathOperator{\Homeo}{Homeo}
\DeclareMathOperator{\Map}{Map}
\DeclareMathOperator{\Mod}{Map}

\DeclareMathOperator{\Fix}{Fix}
\DeclareMathOperator{\Per}{Per}
\DeclareMathOperator{\Stab}{Stab}

\DeclareMathOperator{\Core}{Core}

  \newcommand{\Z}{\mathbb{Z}}
\let\oldmarginpar\marginpar
\renewcommand{\marginpar}[1]{\oldmarginpar{\footnotesize\color{red} #1}}

  \newcommand{\param}{{\mathchoice{\mkern1mu\mbox{\raise2.2pt\hbox{$
  \centerdot$}}
  \mkern1mu}{\mkern1mu\mbox{\raise2.2pt\hbox{$\centerdot$}}\mkern1mu}{
  \mkern1.5mu\centerdot\mkern1.5mu}{\mkern1.5mu\centerdot\mkern1.5mu}}}

\title{An algorithm to decide if an outer automorphism is geometric}
\author{Edgar A. Bering IV}
\address{Department of Mathematics \& Statistics, San José State University}
\email{edgar.bering@sjsu.edu}
\author{Yulan Qing}
\address{Department of Mathematics, University of Tennessee at Knoxville}
\email{yqing@utk.edu}
\author{Derrick R. Wigglesworth}
\address{http://www.drwiggle.com}
\email{derrick.wigglesworth@gmail.com}


\begin{document}

\begin{abstract}
An outer automorphism of a free group is geometric if it can be represented by
a homeomorphism of a compact surface. For irreducible outer automorphisms, Bestvina and Handel gave an algorithmic
characterization of geometricity using relative
train tracks in 1995. The general case remained open, with a few partial results appearing subsequently. In this paper we give a complete resolution to the problem: an algorithm that can decide if a general outer automorphism is geometric. 
The algorithm is constructive and produces a realizing surface homeomorphism if one exists. We make use of 
 advances in train-track theory, in conjunction with the
Guirardel core of tree actions and Nielsen-Thurston theory for surfaces.
\end{abstract}

\maketitle

\section{Introduction}

A free group $\F$ of rank $\rk$ can be realized as the fundamental group of a
surface $\Sigma$ with boundary, and homeomorphisms of $\Sigma$ induce outer
automorphisms on $\F_{\rk}$. Any outer automorphism $\phi\in\Out$ that arises
from this construction is \emph{geometric} (see \Cref{geometric-definition-section}). Bestvina and Handel gave an algorithm to determine when an irreducible outer automorphism is geometric early in the history of train-track theory~\cite{BH95}, but a general algorithm has eluded the community. 
We produce an algorithm to decide whether or not a general outer
automorphism is geometric. 
Our algorithm is constructive, in the sense that if a realization exists there is a procedure to compute the surfaces together with the homeomorphism.

\begin{introthm}
	Let $\phi \in \Out$ be an outer automorphism. There exists an algorithm
	that decides if $\phi$ is geometric and a procedure to compute a
	realization of $\phi$ as a surface homeomorphism if one exists.
\end{introthm}

Throughout the paper we adopt the notational convention that $\F$ is a finite-rank free group of rank $\rk$, which we typically suppress.

\subsection*{3--Manifold free-by-cyclic groups}
   A finitely-generated, free-by-infinite-cyclic group is a group $G$ which admits an
   epimorphism onto $\Z$ such that the kernel is a finitely generated
   non-cyclic free group. For brevity we refer to such groups
   as free-by-cyclic groups. If $G$ is a free-by-cyclic group, then
   the epimorphism $G \rightarrow \Z$ splits, and we can thus write $G$ as a
   semi-direct product \[\F \rtimes_{\phi} \Z\]
    where  $\F$ is a non-cyclic free group and $\phi \from \F \to
    \F$ is an isomorphism.  If $\phi$ is geometric and is realized by a
    surface homeomorphism $h\from \Sigma \to \Sigma$, then $G$ is the
    fundamental group of the mapping torus of $h$, a 3--manifold with
    non-trivial torus boundary. Not every free-by-cyclic group is the
    fundamental group of a 3-manifold. However, it is a consequence of
    Stallings' fibering theorem~\cite{hempel}*{Chapter 11} that if a
    free-by-cyclic group $G$ is the fundamental group of a 3-manifold then it
    is a mapping torus of some surface homeomorphism realizing the monodromy.
    As a consequence, our algorithm can detect 3-manifold groups among all 
    free-by-cyclic groups.
   
   \begin{introcor}
    There exists an algorithm to decide whether a given free-by-cyclic group is a 3-manifold group.
   \end{introcor}

\subsection*{History}

The question of characterizing geometric outer automorphisms has its roots in the Dehn-Nielsen-Baer theorem for closed orientable surfaces, which states that \emph{any} outer automorphism of the fundamental group of a closed orientable surface is geometric in our sense. The techniques in the proof can also be used to show that for free groups of rank $\rk \le 2$ every outer automorphism of $\F_{\rk}$ is geometric.

The story is much longer in rank greater than three. An outer automorphism of $\F$ is \emph{fully irreducible} if no conjugacy
class of a nontrivial proper free factor is periodic. An outer automorphism
of $\F$ is \emph{ irreducible} if it fixes no conjugacy class of a
nontrivial proper free factor. Bestvina and Handel characterized irreducible
automorphisms that are geometric \cite{BH95}. Given an outer automorphism
$\phi$, Bestvina and Handel's characterization finds a
periodic conjugacy class $[c]$. This conjugacy class represents a boundary curve of
the surface, and a surface where $\phi$ can be realized as a homeomorphism can
be obtained from a suitable graph by attaching an annulus along a cycle
representing $[c]$. This idea is one of the building blocks of our algorithm.

Using the core of tree actions, Guirardel gave a limiting characterization of a
geometric fully irreducible outer automorphism. The projectivisation of the space of actions of $\F$ on very-small
$\mathbb{R}$-trees is known as the compactification of Culler-Vogtmann outer
space, $\overline{CV}$~\cites{cv86,very-small-1,very-small-2}. For 
a fully irreducible automorphism $\phi$ acting on $\overline{CV}$ there are
unique forward and backward limit points, with representatives $T^{+}$ and $T^{-}$. Guirardel proved that if both $T^{+}$ and $T^{-}$ are
dual to measured foliations on surfaces, then $\phi$ is realized by a
pseudo-Anosov homeomorphism of a punctured surface~\cite{guirardelcore}*{Corollary
9.3}. Moreover, the universal cover of the realizing surface, equipped with the
singular Euclidean metric coming from the transverse foliations for the
pseudo-Anosov is isometric to the core of the two tree actions. The core (though not of
limit points) also plays a role in our algorithm.

There has been relatively little progress in understanding the geometricity of
elements of $\Out$ beyond the irreducible setting.
Ye treated the case of polynomially growing outer automorphisms with an
analysis of generalized Dehn twists~\cite{ye-pg}. For any given polynomially
growing automorphism $\phi \in \Out$, Ye provides an explicit power $t(\rk)$ and
an algorithm to decide whether $\phi^{t(\rk)}$ is geometric or not. 

\subsection*{Key tools and technical advances} 
Our algorithm proceeds in two stages. First, we pass to what is
known as a \emph{rotationless power} to avoid the complications of finite order
behavior. \Cref{alg:rotationless} uses Feighn and Handel's algorithmic CT
representatives to decide if a rotationless outer automorphism is geometric.
Once this is done, \Cref{alg:final} uses the Guirardel core and Krstic,
Lustig, and Vogtmann's equivariant Whitehead algorithm to decide if a given
root of a geometric rotationless outer automorphism is again geometric. To
combine these tools we requires certain extensions of ideas in the literature
that we highlight here, in hopes they have broader application.

\subsubsection*{Surface data of CT representatives}
CTs are graph map representatives of outer automorphisms designed by Feighn and Handel to
satisfy the properties that have proven most useful for investigating elements
of $\Out$~\cites{FH:Abelian, FH:RecogThm}. Moreover, given a rotationless outer automorphism, a CT representative can be
produced algorithmically~\cite{FH:CTs}. \Cref{alg:rotationless} uses such a representative as
a starting point. A wealth of data can be derived from a CT representative,
including surfaces and pseudo-Anosov homeomorphisms that represent surface-like 
exponentially growing behavior of the outer
automorphism~\cite{HandelMosher}*{Chapter I.2}. We add to this
literature with \Cref{cor:godgivensurfacepiece}, which roughly says that all of
the exponentially growing surface data of a CT is a property of the outer
automorphism alone and does not depend on the choice of representative. This extends work of Handel and Mosher, as detailed in
\Cref{section:GeometricModels}.

Using this surface data, the final step of \Cref{alg:rotationless} is to verify
that the candidate boundary curves can all be glued together into a surface.
This is accomplished using the Whitehead algorithm.

\subsubsection*{Cores and spines}
 Guirardel generalized a group theoretic perspective on the intersection of curves on surfaces to the
 \emph{intersection complex} or \emph{core} of a pair of minimal group actions on a tree \cite{guirardelcore}.
   In the context of $\Out$, suppose $\phi$ is a geometric outer automorphism, and $T$
   is the universal cover of a spine for a surface $\Sigma$ where $\phi$ is
   realized as a homeomorphism. Twisting the action by $\phi$ gives two $\F$
   actions on trees, and Guirardel proved that the intersection core
  $\Core(T, T\phi)$ embeds in the universal cover $\Sigma$. We extend this characterization to what
  we call \emph{partially geometric} outer automorphisms. An outer automorphism
  is partially geometric if it can be realized as a homotopy equivalence of a
  surface $h\from \Sigma \to \Sigma$ which restricts to a homeomorphism on an invariant subsurface $Q\subseteq
  \Sigma$. Subject to some technical conditions, we prove a relative version of
  Guirardel's result in \Cref{prop:all-surface-all-spine}. Roughly, given a
  spine $K$ for $Q$, where $K$ is a subgraph of a spine for $\Sigma$, the portion of
  the intersection core $\Core(T, T\phi)$ projecting to $K$ embeds in some copy of
  the universal cover of $Q$ inside the universal cover of $\Sigma$. That is,
  this ``surface detection'' is a local property.

  Behrstock, Bestvina, and Clay give an algorithm for computing a fundamental
  domain for $\Core(T, T\phi)$ for any outer automorphism $\phi$, so our local
  condition is algorithmic~\cite{behrstockbestvinaclay}. This is used by our general case algorithm,
  \Cref{alg:final}, to verify that an outer automorphism with geometric
  rotationless power is compatible with the surface data of that rotationless
  power. As in the rotationless case, the final verification is completed with
  the Whitehead algorithm. However, due to finite-order behavior we require the
  equivariant Whitehead algorithm of Krstic, Lustig, and
  Vogtmann~\cite{krstic-lustig-vogtmann}.

\subsection*{Finitely generated subgroups of $\Out$} 
One might hope that our algorithm can be extended to decide, given a finitely
generated $ H = \langle \phi_1,\ldots,\phi_k\rangle \le\Out$, whether there is a surface
$\Sigma$ where every element of $H$ is realized as a homeomoprhism;
equivalently if $H \le \Map(\Sigma)$ for the natural inclusion of the
mapping class group of $\Sigma$ into $\Out$. 

This is sometimes possible using
the techniques in this paper. If one of the generators $\phi_i\in H$
is realized by $g_i\from \Sigma\to \Sigma$ and every component of the Thurston
normal form of (a rotationless power of) $g_i$ is pseudo-Anosov, then the
surface $\Sigma$ can be computed by our algorithm and it is the unique surface
where $\phi_i$ is realized. In this case, the subgroup
$H$ is simultaneously realized if and only if each generator $\phi_j$ is in
$\Map(\Sigma)$ for this specific surface $\Sigma$. \Cref{partially-geometric-test} provides an algorithm
to test if $\phi_j$ is geometric on $\Sigma$ for $j\neq i$, so in this case
there is an algorithm.

On the other hand, if $\phi\in\Out$ is geometric, non-identity, infinite order,
and the Thurston normal form of a homeomorphism $g\from \Sigma\to\Sigma$ realizing $\phi$ restricts to the identity
on a non-annular subsurface $R\subseteq \Sigma$, then there are many different
homeomorphism types of surface where $\phi$ is realized. Worse, the
identification of $\Sigma$ with $\F$ can be changed in non-geometric ways on
the subgroup corresponding to $\pi_1(R)$, so there is an infinite collection of
non-equivalent marked surfaces where $\phi$ is realized by a homeomorphism. For a single
outer automorphism we are able to analyze this phenomenon well enough to take
roots in \Cref{finalalgo}. In general this ambiguity is an obstacle to applying the techniques
of our algorithm to a finite generating set.

The structure of finitely generated subgroups in $\Out$ has a rich theory, as
developed in Handel and Mosher's monograph~\cite{HandelMosher}. Given a
finitely generated subgroup $H\le \Out$, after passing to a finite-index
subgroup there is an invariant filtration by
free-factor systems and a classification of what behavior can occur while
moving up the filtration~\cite{HandelMosher}*{Theorem D}. Clay and
Uyanik~\cite{clay-uyanik}*{Theorem 6.6} provide the existence of a witness
automorphism $\phi\in H$ that is relatively irreducible with respect to each
``interesting'' filtration step
in the context of Handel and Mosher's subgroup decomposition
theorem~\cite{HandelMosher}*{Theorem D}. A necessary condition for $H$ to be
realized on a surface is that this witness automorphism $\phi$ is geometric.
Our algorithm will then provide a realization $g\from \Sigma\to\Sigma$ of
$\phi$, and relate the supports of the Thurston normal form of $g$ to the invariant filtration
for $H$. Using the procedures in \Cref{sec:partial-to-geometric} we can
determine if each generator can be realized as a homeomorphism on a
modification of $\Sigma$. However it is not immediately clear that this can be
done simultaneously without significant technical headache. Unfortunately, to
obtain
an algorithm one would need an algorithmic version of the subgroup
decomposition results mentioned in this paragraph, which we are not aware of.

\subsection*{Organization of the paper}

This paper draws on a breadth of $\Out$ theory; we endeavor to recall all
relevant definitions with a common notation. This is done in
\Cref{Section:Prelim}. The first part of the paper is devoted to developing the
necessary tools for the rotationless case: \Cref{section:GeometricModels}
introduces the geometric models of EG strata and derives new invariance
results; these results are used in \Cref{sec:rotationless} to give an
algorithm for handling the rotationless case. The second part of the paper
develops the notion of partially geometric outer automorphisms and connects
detecting this notion to Guirardel's core (\Cref{sec:partial,sec:core}).
The general algorithm is given in \Cref{finalalgo} after developing one more
tool for treating some finite order behavior in
\Cref{sec:partial-to-geometric}. Illustrative examples of some of the varied
behavior of different cases in the algorithm are provided in
\Cref{sec:example}.

\subsection*{Acknowledgements}

The authors are grateful to the Fields Institute for hospitality during the
2018 thematic program on ``Teichmüller Theory and its Connections to Geometry,
Topology and Dynamics'', where this work began. E. A. Bering and Y. Qing thank
the American Institute of Mathematics for hospitality where this work was
completed. E. A. Bering was also partially supported Azraeli foundation.
 
\section{Preliminaries}\label{Section:Prelim}

Let $\F$ denote the free group. For the remainder of the paper unless otherwise noted we assume this to be of rank
$\rk \ge 3$. For an element or subgroup of $\F$, we use $[\cdot ]$ to
denote its conjugacy class. 
Let $\Out = \Aut(\F) / \Inn(\F)$ denote the group of outer automorphisms of
$\F$.

\subsection{Graphs, paths, and splittings}

Fix an identification $\F = \pi_1(\mathfrak{R}, \ast)$ for the
wedge of $\rk$ circles (sometimes called the $\rk$-petaled rose) $\mathfrak{R}$. A \emph{marked graph} $G$ is a
graph where each vertex has degree at least three and a homotopy equivalence
$m: \mathfrak{R}\to G$ called a \emph{marking}, which identifies
$\F$ with $\pi_1(G,m(\ast))$. The outer automorphism group of $\F$ acts on the
set of marked graphs on the right by twisting the marking: $(G,m)\mapsto (G,
m\circ\phi)$. A homotopy equivalence $f :  G \to G$
determines an outer automorphism $\phi \in \Out$ by $\phi = [\bar{m}\circ f\circ
m]$ where $\bar{m}$ is a homotopy inverse to $m$. If $f$ sends vertices to vertices and
is an immersion on each edge, then we say $f$ is a \emph{topological representative} of
$\phi$.

A \emph{path} in a marked graph is an isometric immersion $\sigma :  I \to 
G$ of an interval or a constant map, the latter is called a trivial path. Any
map $\sigma:  I \to G$ is homotopic relative to endpoints to a unique path
$[\sigma]$, called its \emph{tightening}.
A \emph{circuit} is an immersion $\sigma :  S^1\to G$, and similarly any
map $\sigma:  S^1\to G$ has a
tightening. Any path or circuit has a decomposition into edges $E_1\ldots
E_{\ell(\sigma)}$ where $\ell(\sigma)$ is the \emph{length} of this
decomposition, and each $E_i$ is an isometry to a given edge. For any path or
circuit $\overline{\sigma}$ denotes $\sigma$ with reversed orientation.  A
finite graph is a \emph{core graph} if every edge is contained in some
circuit, and every finite graph deformation retracts onto a unique core
subgraph.

\subsection{Lines and laminations}

 For a free group $\F$ the boundary pairs
\[ \tilde{\mathcal{B}}(\F) = (\partial\F \times \partial \F \setminus \Delta) /
(\mathbb{Z}/2\mathbb{Z}) \]
is the set of unordered pairs of distinct boundary points of $\partial \F$,
given the topology induced by the usual topology on $\partial \F$. The space of
\emph{abstract lines} in $\F$ is the quotient $\mathcal{B}(\F) =
\tilde{\mathcal{B}}(\F) / \F$ of boundary pairs by the action of $\F$, equipped
with the quotient topology. For a line $\ell\in\mathcal{B}(\F)$ a \emph{lift} is any
point $\tilde{\ell}\in\tilde{\mathcal{B}}(\F)$ that projects to $\ell$.
Every conjugacy class $[w]\in \F$ determines a
well-defined \emph{axis} characterized as the image of a boundary pair fixed by
a representative $w$. For a finite rank subgroup $K \le \F$ there is a natural
inclusion $\tilde{B}(K) \subseteq \tilde{B}(\F)$ which induces an inclusion
$\mathcal{B}(K) \subseteq \mathcal{B}(\F)$. The image of the latter map depends
only on the conjugacy class $[K]$. 

A closed subset $\Lambda \subseteq \mathcal{B}(\F)$ is called a
\emph{lamination} of $\F$. A line $\ell \in \Lambda$ is a \emph{leaf} of the
lamination and $\ell$ is a \emph{generic leaf} if the closure of $\ell$ equals
$\Lambda$. A lamination \emph{fills} a subgroup $K\le \F$ if it is not carried
by any proper free factor system of $K$. It is a consequence of the Shenitzer
and Swarup theorems on cyclic splittings of free groups that if a lamination is
carried by the vertex group of a cyclic splitting then it is not
filling~\cites{Shenitzer,Swarup}.

Associated to each $\phi\in\Out$ is a finite
$\phi$-invariant finite set of laminations $\mathcal{L}(\phi)$, called the set of
\emph{attracting laminations}, and a bijection $\mathcal{L}(\phi) \to
\mathcal{L}(\phi^{-1})$, a pair of laminations $\Lambda^+, \Lambda^-$ related
by this bijection are a \emph{dual lamination pair} for $\phi$ and the set of
dual lamination pairs is denoted $\mathcal{L}^{\pm}(\phi)$.

Given a finite graph $G$ the \emph{space of lines} in $G$, denoted
$\mathcal{B}(G)$ is the set of all isometric immersions $\ell \from  \mathbb{R}
\to G$. The space of lines is topologized by the \emph{weak topology} which has
a basic open set $V(G,\alpha)$ for each path $\alpha$ consisting of lines that
have $\alpha$ as a subpath. The marking $m$ of a marked
graph identifies $\mathcal{B}(G)$ and $\mathcal{B}(\F)$ via homeomorphism,
induced by the identification of $\partial\F$ with the ends of $\tilde{G}$. An
outer automorphism $\phi$ induces a self-homeomorphism 
\[ \phi_\# \from \mathcal{B}(\F) \to \mathcal{B}(\F) \]
 and if $f$ is a topological realization of
$\phi$ the composition 
\[\mathcal{B}(\F)\cong\mathcal{B}(G) \overset{f_\#}{\to}
\mathcal{B}(G)\cong\mathcal{B}(\F)\]
is equal to $\phi_\#$. For a marked graph $G$, we say a line or lamination
in $\mathcal{B}(G)$ \emph{realizes} the corresponding line or lamination in
$\mathcal{B}(\F)$. 

\subsection{Free factor systems and supports}
A \emph{free factor system} of $\F$ is a finite
collection of proper free factors of $\F$ 
such that there exists a free factor decomposition 
\[\F = A_1\ast\cdots\ast A_k\ast B.\]
 A free factor
system $\mathcal{F}$ is the collection
\[
\mathcal{F} = \{ [A_1],\ldots, [A_k] \}, \text{ where } k\geq 0.
\]
We say that a free factor system \emph{carries} a conjugacy class $[K]$ if there is some
$[A]\in\mathcal{F}$ such that $K\le A$. Free factor
systems are partially ordered by extending the carrying relationship
$\mathcal{F}_1\sqsubset\mathcal{F}_2$ if $\mathcal{F}_2$ carries each
$[A]\in\mathcal{F}_1$. 

A subgroup $K$ \emph{carries} a set of lines
$W$ if $W\subset\mathcal{B}(K)$, and a free factor system $\mathcal{F}$
\emph{carries} $W$ if each element of $W$ is carried by some
$[A]\in\mathcal{F}$. For a set $\mathcal{C}$ of conjugacy classes of
subgroups, conjugacy classes of elements and lines we define the \emph{free factor support}
$\mathcal{F}_{supp}(\mathcal{C})$ is the $\sqsubset$-minimal free factor system
carrying each element of $\mathcal{C}$. Such a support exists and is
unique~\cite{HandelMosher}*{Fact I.1.10}.

\subsection{$\Out$ and surface homeomorphisms}\label{geometric-definition-section}

When $\Sigma$ is a compact surface the mapping class group $\Map(\Sigma)$ is the set of isotopy classes
of homeomorphisms of $\Sigma$. We explicitly include homeomorphisms that
permute the boundary and orientation-reversing homeomorphisms, in order to
fully capture the behavior exhibited by geometric outer automorphisms.

Similar to a marked graph, a marked surface is a compact surface $\Sigma$ and a
homotopy equivalence $m \from \mathfrak{R} \to \Sigma$ called a \emph{marking},
identifying $\F$ and $\pi_1(\Sigma, m(\ast))$. As with graphs we will suppress
the marking unless necessary.

For algorithmic purposes, for each rank $n$ we fix a finite list
$S_{n,1},\ldots, S_{n,k}$ of \emph{standard surfaces} such that
$\mathfrak{R}\subset S_{n,i}$ and $S_{n,i}$ deformation retracts onto
$\mathfrak{R}$; one for each homeomorphism type of surface with fundamental
group $\F_n$ (including non-orientable surfaces). Observe that if $(\Sigma, m)$
is a marked surface homeomorphic to some $S_{n,i}$ then there is an outer
automorphism $\phi$ such that $(S_{n,i},\phi)$ is marked-homeomorphic to
$(\Sigma, m)$.

An outer automorphism $\phi$ is \emph{geometric} if there exists a marked
surface $\Sigma$ and a homeomorphism $g : 
\Sigma \to \Sigma$ with $g_{*}$ a representative of $\phi$. It is important to
note $\Sigma$ need not be orientable and $g$ may not restrict to the identity
on $\partial \Sigma$. In addition to this standard notion, our algorithm will
need a partial notion in intermediate steps. A connected subsurface
$Q\subseteq\Sigma$ is \emph{essential} if it has infinite fundamental group and
is $\pi_1$ injective. A general subsurface $Q\subseteq \Sigma$ is essential if
each connected component is
essential and no component of the complement $\Sigma \setminus Q$ is an annulus.

\begin{definition}
	An outer automorphism $\phi$ is \emph{partially geometric} on a
	marked surface $\Sigma$ with respect to a (possibly disconnected) essential subsurface  $Q \subset \Sigma$
 if $\phi$ is realized by a homotopy equivalence $g \from \Sigma \to \Sigma$ such that:
	\begin{itemize}
		\item The decomposition of $\Sigma$ into $Q$ and $\Sigma
			\setminus Q$ is $g$ invariant,
		\item $g$ restricted to $Q$
 is a homeomorphism.  
	\end{itemize}
	Any homotopy equivalence with these properties is a \emph{geometric
	witness} for $\phi$.
\end{definition}

\begin{remark}
	Note that the notion of a partially geometric outer automorphism  is
	stronger than that of an outer automorphism having geometric strata,
	because the complementary subsurface is preserved up to homotopy.
\end{remark}

As we are working with un-oriented compact, connected surfaces and not requiring
that our homeomorphisms restrict to the identity on the boundary, we quote the
following two classical surface theorems in full with references to the
specific formulations used.

\begin{theorem}[Dehn-Nielsen-Baer Theorem]\label{dehn-nielsen-baer}
Let $f \from  \Sigma \to \Sigma'$ be a homotopy equivalence between compact,
connected surfaces with $\chi(\Sigma) = \chi(\Sigma') < 0$. Assume that $f$
restricts to a homeomorphism $\partial \Sigma = \partial \Sigma'$.  Then $f$ is
homotopic (relative to $\partial\Sigma$) to a homeomorphism $g\from  \Sigma
\simeq \Sigma'$.
\end{theorem}

\begin{proof}
Fujiwara~\cite{Fujiwara}*{\S 3} records the non-orientable case of
Maclachlan and Harvey's generalization of the Dehn-Nielsen-Baer
theorem for so-called type-preserving outer automorphisms of a Fuchsian
group~\cite{MaclachlanHarvey}*{Theorem 1}. In the setting of homotopy
equivalences of compact surfaces with negative Euler characteristic,
type-preserving reduces to preserving the set of boundary conjugacy classes.
Thus, the hypothesis on $f$ is exactly that it induces a type-preserving isomorphism
of $\pi_1(\Sigma)\to \pi_1(\Sigma')$. In turn this implies the two surfaces are
homeomorphic and $f$ is homotopic to a homeomorphism by Fujiwara's formulation
of the Dehn-Nielsen-Baer theorem.
\end{proof}

In this article we will need to apply the Dehn-Nielsen-Baer theorem to
restrictions of maps to disconnected subsurfaces. To do so we need the
following standard result.

\begin{lemma} Suppose $Q, R$ are homeomorphic connected subsurfaces of a
	compact surface $\Sigma$. Let  $\phi
	\from \Sigma \to \Sigma$ be a
	homotopy equivalence preserving the decomposition of $\Sigma$ into
	$\Sigma\setminus (Q\cup R)$ and $Q\cup R$, and that $\phi(Q) = R$ . For any
	homeomorphism $h \from R\to Q$, the composition $h \circ \phi \from
	Q\to Q$
is homotopic to a homeomorphism if and only if the restriction $\phi \from Q\to
	R$ is homotopic to a homeomorphism.
\end{lemma}

\begin{proof}
	One direction is clear. Suppose that $h\circ \phi \from Q\to Q$ is
	homotopic to a homeomorphism and let $H$ be the homotopy. The
	composition $h^{-1}H$ is the desired homotopy.
\end{proof}

Given a surface $\Sigma$, let 
$ \Map(\Sigma) $
denote the equivalence classes of orientation-preserving homeomorphisms on $\Sigma$, up to isotopy. 
Let $[g]\in \Map(\Sigma)$ be a mapping class represented by a homeomorphism $g$. A \emph{reducing system} for
$g$ is a collection of disjoint, simple closed curves $\mathcal{C}$ such that,
$g(\mathcal{C}) = \mathcal{C}$. A
mapping class is \emph{reducible} if it has a representative with a reducing system and
\emph{irreducible} otherwise. The reduction of $[g]$ along
$\mathcal{C}$ is the image of $[g]$ under the natural homomorphism
$\Map(\Sigma)\to \Map(\Sigma_\mathcal{C})$, where $\Sigma_\mathcal{C}$ is
the complement of a system of disjoint open neighborhoods of $\mathcal{C}$.
The \emph{canonical reduction system} for $g$ is the intersection of all
inclusion-wise maximal reduction systems. If $\mathcal{C}$ is the canonical
reduction system, $\mathcal{C}$ has a power such that $\bar{g}$ fixes each
component and each restriction is finite-order or irreducible.

\begin{theorem}[Thurston Normal Form~\cites{FLP, Wu:NF}]
	If $[g]\in \Map(\Sigma)$ is a mapping class of a compact, connected
	surface $\Sigma$, then
	either $[g]$ is irreducible or there is a representative $g$ and a
	canonical reduction system $\mathcal{C}$ fixed by $g$.
\end{theorem}

It is a standard consequence of this normal form that after passing to a
sufficient power any mapping class can be factored as a product of Dehn twists
about the reducing curves and the images of irreducible mapping classes on the
complement under inclusion; see \ref{NoPeriodicBehavior} for the precise formulation.

\subsection{Automorphisms and lifts}
In this section we introduce all the notations and facts needed for actions and dynamics on $\F$ and $\partial \F$.
As in the previous sections $\mathfrak{R}$ is a rose with base-point $(\ast)$ and  $G$ is a marked rose.  Let $\phi \in \Out$ and  let $f \from G\to G$ be a topological representative of $\phi$. Let $b =
m(\ast) \in G$ be the base-point in $G$. The set of paths $\sigma$ from $b$ to $f(b)$
determines a bijection between the automorphism lifts of $\phi$ and lifts of
$f$ to $\widetilde{G}$. This bijection can be seen without reference to a
base-point as we now detail.

Let  $\Phi$ be an automorphism of $\F$ that represents  $\phi$.  
Each $u \in \F$ acts on $\widetilde{G}$ by
the covering transformation $\tau_u$. Additionally, there are a pair of points
\[\{ u^+_\infty, u^-_\infty  \} \subset \partial \F\] that are respectively the limits of
positive and negative powers of $u$. The marking identifies $\partial \F$ with
$\partial \widetilde{G}$. The line in $\widetilde{G}$ whose ends converge to
$u^+_\infty$ and $u^-_\infty$ is called the \emph{axis} of $\tau_u$.
The bijection between lifts and
automorphisms pairs $\tilde{f}$ to $\Phi$ when 
\[\tilde{f}\tau_u = \tau_{\Phi(u)}\tilde{f} \text{ for all }u \in \F.\]

An automorphism $\Phi$ of $\F$ induces a homeomorphism $\hat{\Phi}$ on $\partial \F$. The action of the group $\Aut(\F)$ on $\F$ has a continuous extension to an action on the Gromov compactification 
$ \F \cup \partial\F$: Given $ \Phi \in \Aut(\F)$, let $\hat{\Phi} :  \partial\F  \to  \partial\F$ denote its continuous extension, and let $\Fix(\hat{\Phi}) \subset \partial\F$ be the set of fixed points of $\hat{\Phi}$. Let $\Fix(\Phi) < \F$ denote the subgroup of elements fixed by $\Phi$.

 Let $\Fix_+(\hat{\Phi})$ denote the set of attractors in $\Fix(\hat{\Phi})$, a discrete subset consisting of points $\xi \in \Fix(\hat{\Phi})$ such that for some neighborhood $U \subset \partial\F$ of $\xi$ we have $\hat{\Phi}(U) \subset U$ and the sequence $\hat{\Phi}^n(\eta)$ converges to $\xi$ for each $\eta \in U$. Let 
 \[\Fix_-(\hat{\Phi}) := \Fix_+(\hat{\Phi}^{-1}) \] denote the set of repellers in $\Fix(\hat{\Phi})$. This gives a decomposition of the fixed set of $\hat{\Phi}$:
\begin{equation}\label{disj}
\Fix(\hat{\Phi}) = \partial \Fix(\Phi) \cup \Fix_-(\hat{\Phi}) \cup \Fix_+(\hat{\Phi}).
\end{equation}
In the sequel we are interested primarily in the \emph{non-repelling fixed
points}, \[\Fix_{N}(\hat{\Phi}) = \Fix(\hat{\Phi})\setminus
\Fix_{-}(\hat{\Phi}).\]

We also denote the periodic point set 
\[
\Per(\hat{\Phi}) :=\bigcup_{k \geq 1} \Fix( \hat{\Phi}^k)
\]
and its subsets $\Per_+(\hat{\Phi}), \Per_-(\hat{\Phi}), \Per_N(\hat{\Phi})$ defined by similar unions.

\subsection{Principal lifts and rotationless outer automorphisms}

\begin{definition}\label{principal-lift}
A representative $\Phi\in\Aut(\F)$ of an outer automorphism $\phi\in\Out$ is
\emph{a principal lift} if either:
\begin{itemize}
\item $\Fix_{N}(\hat{\Phi})$ contains at least three points.
\item $\Fix_{N}(\hat{\Phi})$ is a two point set that is neither the set of endpoints of an axis nor the set of endpoints of a lift of a generic leaf of an element of $\Lam(\phi)$.
\end{itemize}
The corresponding lift $\tilde{f}:  \tilde{G} \to \tilde{G}$ of a topological
	representative of $\phi$ is also called
a \emph{principal lift}. The set of principal lifts for $\phi$ is denoted
$P(\phi)$.
\end{definition}

Note that there is a map  $P(\phi)\to P(\phi^k)$ for $k\geq 1$. Two
automorphisms $\Phi_1, \Phi_2$ are \emph{isogredient} if there is a $c\in \F$
such that for the inner automorphism $i_c$, $\Phi_2 = i_c\Phi_1i_c^{-1}$. An
outer automorphism $\phi$ has finitely many iso-gredience classes of principal
lifts~\cite{FH:RecogThm}*{Remark 3.9}.

\begin{definition}\label{rotationless}
An outer automorphism $\phi$ is \emph{(forward) rotationless} if for all $k\geq
1$ the map $P(\phi)\to P(\phi^k)$ is a bijection and $\Fix_N(\Phi) =
\Per_N(\Phi)$ for all $\Phi \in P(\phi)$.
\end{definition}

Rotationless automorphisms are without periodic behavior:  if $\phi$
is rotationless then any $\phi$-periodic conjugacy class, lamination, or free factor
system is in fact $\phi$ invariant~\cite{FH:RecogThm}*{Lemma 3.30}.

\subsection{Principal lifts and the Nielsen approach to $\Map(\Sigma)$}

Suppose now $\phi\in\Map(\Sigma)$ for a compact surface $\Sigma$. There exists
a homeomorphism $g \from \Sigma
\to \Sigma$ representing $\phi$ that preserves both the stable and unstable
geodesic laminations~\cite{CB88}*{Lemma 6.1}. If in addition $g$ preserves each individual principal
region of the stable and unstable geodesic laminations, its boundary leaves,
and their orientations then we say that $g$ is \emph{rotationless}; A
rotationless power of $g$ exists because there are only finitely many
principal regions and boundary leaves. We say that $\phi$ is rotationless if it
has a rotationless representative. If $g$ is rotationless and if
$\widetilde{g} \from \widetilde{\Sigma} \to \widetilde{\Sigma}$ is a lift of
$g$, we say that $\widetilde{g}$ is an \emph{s-principal lift} ($s$ stands for
stable) if there exists a principal region $R^s$of the stable geodesic
laminations such that $\widetilde{g}$ preserves $\widetilde{R}$ and
preserves some (every) boundary leaf of $\widetilde{R}$. 
\begin{observation}
	As shown by Handel and Mosher the principal lifts in the $\Out$ and 
	$\Map(\Sigma)$ sense agree \cite{HandelMosher}*{Proposition I.2.12}.  
\end{observation}

For geometric outer automorphisms, the boundary curves of the surface can be
detected by the dynamics of principal lifts.

\begin{proposition} \label{prelim-principal-region-structure-theory}
Suppose $g$ is a pseudo-Anosov diffeomorphism of a compact surface $\Sigma$. Then the following are equivalent:
\begin{enumerate}
\item $c \in \pi_1(\Sigma)$ is a boundary conjugacy class.
\item There is a principal lift of $\tilde{g}$ such that the endpoints of $c$ are a non-repelling fixed point for the action of $\tilde{g}$ on $\partial \mathbb{H}^2$.
\item The endpoints of $c$ are non-repelling fixed points for the action of  $\tilde{g}_*$ on $\partial \pi_1(\Sigma)$.
\end{enumerate}
\end{proposition}
\begin{proof}
The equivalence of (2) and (3) is standard \cite{Nie86}, since there's a
continuous embedding of $\partial \pi_1( \Sigma) \in \partial \mathbb{H}^2$
that respects the action. The equivalence of (1) and (2) is established with
care by Handel and Mosher~\cite{HandelMosher}*{Proposition 2.12}. 
\end{proof}

\subsection{Train tracks, splittings, and CTs}

Our analysis of individual outer automorphisms requires the consequences of
particularly refined topological representatives, known as \emph{completely
split improved relative train-track maps (CTs)}, introduced by Feighn and
Handel~\cite{FH:RecogThm}. In lieu of the precise definition, the statement of which
requires significant structure that we do not use directly, we introduce the
parts of the definition needed for this paper, and indicate whenever a
particular quoted lemma for topological representatives applies only to CTs.

A \emph{filtration} of a topological representative $f:  G\to G$ of an outer
automorphism $\phi$ is a choice of an $f$-invariant chain of subgraphs
\[\emptyset = G_0 \subset G_1 \subset \cdots \subset G_R = G.\] Associated to a
filtration of a marked graph $G$ is a nested free factor system 
\[\mathcal{F}_i = [\pi_1(G_i)]\]
that comes from the
fundamental group of each connected component of $G_i$. The action of
$f$ on the edges of $G$ determines a square matrix known as the
\emph{transition matrix}. The set $H_i =
\overline{G_i\setminus G_{i-1}}$ is the $i^{th}$ stratum of $G$, and the
sub-matrix of $M$ corresponding to the rows and columns indexed by $H_i$ is
denoted $M_i$. For the topological representatives in this article, $M_i$ will
be either a zero matrix, a $1\times 1$ identity matrix, or an irreducible
matrix with Perron-Frobenius eigenvalue $\lambda_i > 1$. We refer respectively
to the stratum $H_i$ as a \emph{zero}, \emph{non-exponentially growing (NEG)},
or \emph{exponentially growing} stratum. Any stratum that is not a zero stratum
is \emph{irreducible}.

If $f:  G\to G$ is a topological representative and $\sigma$ a path or
circuit in $G$ define $f_\#(\sigma) = [f(\sigma)]$. A decomposition of $\sigma
= \sigma_1\cdot\sigma_2\cdots\sigma_n$ into subpaths is a \emph{splitting for
$f$} if \[f_\#^k(\sigma) = f_\#^k(\sigma_1)\cdot f_\#^k(\sigma_2)\cdots
f_\#^k(\sigma_n) \text{ for all }k\ge 1;\]
 we reserve $\cdot$ to denote splittings and
use adjacency for plain concatenation. A path or circuit $\sigma$ is called a
\emph{periodic Nielsen path} if $f_\#^k(\sigma) = \sigma$ for some $k\geq 1$,
if $k=1$ then $\sigma$ is a \emph{Nielsen path}. A Nielsen path is
	\emph{indivisible} if it is not a concatenation of non-trivial Nielsen
	paths. In general a closed path $\sigma$ is \emph{root-free} if $\sigma
	\neq \gamma^k$ for a closed path $\gamma$. If $w$ is a closed, root-free Nielsen path and $E$ is an edge
	such that $f(E) = Ew^k$ then we call $E$ a \emph{linear edge} and $w$
	the axis of $E$. If $E$ and $E'$ have a common axis $w$ where $k \neq
	k'$ and $k, k' > 0$, then any path of the form $Ew^\ast \overline{E}'$ is an
	\emph{exceptional path}.
	
For an EG stratum $H_r$ a nontrivial path $\sigma$ in $G_{r-1}$ with endpoints
in $H_r \cap G_{r-1}$ is a \emph{connecting path} for $H_r$. A path $\sigma$
contained in a zero stratum is \emph{taken} if there is an edge $E$ in
an irreducible stratum and $k\geq 1$ such that $\sigma$ is a maximal subpath of
$f_\#^k(E)$ contained in that zero stratum. A non-trivial path or circuit
$\sigma$ is \emph{completely split} if there is a splitting $\sigma =
\sigma_1\cdot\sigma_2\cdots\sigma_k$ where each $\sigma_i$ is one of the
following:  a single edge in an irreducible stratum, an indivisible Nielsen
path, an exceptional path, or a maximal connecting taken path in a zero
stratum. Each piece $\sigma_i$ is referred to as a \emph{splitting unit} of $\sigma$.

\begin{definition}[\cite{FH:RecogThm}*{Definition 4.7}]
	A filtered topological representative \[f \from  G\to G\] of a rotationless
	outer automorphism $\phi$ with filtration
	\[\emptyset = G_0\subset G_1 \subset \cdots \subset G_R = G\] is a
	\emph{completely split improved relative train-track map (CT)} if it is
	a \emph{relative train track map} satisfying
	\begin{description}
		\item[(Completely Split)] For every edge $E$ in an irreducible
			stratum $f(E)$ is completely split; and for every taken
			connecting path $\sigma$ in a zero stratum
			$f_\#(\sigma)$ is completely split.
		\item[(NEG Nielsen Paths)] If the highest edges in an
			indivisible Nielsen path $\sigma$ belong to an NEG
			stratum then there is a linear edge $E$ and a closed
			root-free Nielsen path $w$ such that $f(E) = Ew^d$ for
			$d\neq 0$ and $\sigma = Ew\overline{E}$.
		\item[(Other Good Stuff)] The requirements of the cited
			definition due to Feighn and Handel, which will not be
			directly appealed to in this paper, but are necessary
			for the quoted consequences.
	\end{description}
\end{definition}

For a rotationless outer automorphism $\phi$, the EG strata of the CT are
intimately related to the set of attracting laminations.
 Specifically, there is a bijection
between $\mathcal{L}(\phi)$ and EG strata of a CT~\cite{HandelMosher}*{Fact I.1.55}:
given any
lamination $\Lambda\in\mathcal{L}(\phi)$,  we can associate it with the stratum $i$ where 
\[\mathcal{F}_{supp}(\Lambda)\not\sqsubset
\mathcal{F}_{i-1} \text{ and }\mathcal{F}_{supp}(\Lambda)\sqsubset
\mathcal{F}_i. \]
In
this case we say $\Lambda$ is the lamination associated to the stratum $H_i$.

While the definition provides complete splittings only for edges and taken
connecting paths, iteration suffices to obtain complete splittings of any path.

\begin{lemma}[\cite{FH:RecogThm}*{Lemma 4.25}]
	If $f:  G\to G$ is a CT and $\sigma \subset G$ is a path with
	endpoints at vertices then $f_\#^k(\sigma)$ is completely split for all
	sufficiently large $k$.
\end{lemma}

In order to analyze certain CTs we make routine use of the ``moving up through
the filtration'' lemma, which describes how the strata change as we move from
filtration element that is a core graph to the next filtration element that is
a core graph.

\begin{lemma}[\cite{FH:Abelian}*{Lemma 8.3}]
\label{moving-up-through-the-filtration}
Suppose $f\from G\to G$ is a CT. Fix a filtration level $r$ such that $G_r$ is core, and let $s > r$ be the
smallest integer such that $G_s$ is a core graph. Then there are two possible
cases:
\begin{description}
\item[$H_s$ is NEG] In this case either $s=r+1$ or $s=r+2$ and $H_{r+1}$ is
also NEG.
\item[$H_s$ is EG] In this case there exists $r\leq u < s$ such that
\begin{itemize}
\item For each $r < j \leq u$ the stratum $H_j$ is a single non-fixed edge with
terminal vertex in $G_r$.
\item For each $u < j < s$, $H_j$ is a zero stratum enveloped by $H_s$.
\end{itemize}
\end{description}
\end{lemma}

\subsection{Computing with CTs}

Feighn and Handel introduced an algorithm that produces a CT for any
rotationless outer automorphism, which additionally can start from any
prescribed filtration as input. This algorithm is instrumental in our
algorithm.

\begin{theorem}[\cite{FH:CTs}*{Theorem 1.1}]
	There is an algorithm, refered to as \textsc{ComputeCT}, that takes as input a rotationless $\phi \in
	\Out$ and a nested sequence $\mathcal{F}_0 \sqsubset
	\mathcal{F}_1\sqsubset\cdots\sqsubset\mathcal{F}_k$ of $\phi$-invariant free
	factor systems and produces a filtered topological representative
	$f:  G\to G$ such that every non-empty $\mathcal{F}_i =
	\mathcal{F}(G_r)$ for a core filtration element $G_r$ and $f$ is a CT.
\end{theorem}

In addition to an algorithm for computing a CT, Feighn and Handel demonstrate
that many properties of a CT are computable from the CT. Specifically, 
there are algorithms, referred to as \textsc{ZeroStrata?}, \textsc{NonlinearNEG?}, and
\textsc{NongeometricEG?}, that take as input a CT $f$ and decide respectively if
$f$ has zero, nonlinear NEG, or nongeometric EG strata~\cite{FH:CTs}.

\subsection{The (equivariant) Whitehead algorithm}

In its most general form, the classical Whitehead algorithm is a procedure for understanding orbit-equivalence of tuples of conjugacy classes under the action of $\Out$. We will make repeated use of this formulation, as recorded in Lyndon and Schupp.

\begin{theorem}[\cite{lyndon-schupp}*{Proposition 4.21}]\label{whitehead}
Let $\mathcal{C}$ denote the set of conjugacy classes of $\F$ with its natural
	$\Out$ action. Let $\mathcal{C}^{n}$ denote the product of $n$ copies
	of $\mathcal{C}$ with diagonal $\Out$ action. Then there is an
	algorithm which takes as input $c, c'\in\mathcal{C}^n$ and
	either produces $\phi\in\Out$ such that $\phi(c) = c'$ or \texttt{No} if no such automorphism exists.
\end{theorem}

Krstic, Lustig, and Vogtmann, in their solution to the conjugacy problem for roots of Dehn twists in $\Out$, gave an equivariant generalization of the Whitehead algorithm that respects finite subgroups of $\Out$, which we will also need.

\begin{theorem}[\cite{krstic-lustig-vogtmann}*{Theorem 1.1}]\label{klv-whitehead}
Let $\mathcal{C}$ denote the set of conjugacy classes of $\F$ with its natural
	$\Out$ action. Let $\mathcal{C}^{n}$ denote the product of $n$ copies
	of $\mathcal{C}$ with diagonal $\Out$ action. Let $G$ be a finite
	group. Then there is an algorithm which takes as input $c,
	c'\in\mathcal{C}^n$ and homomorphisms $\alpha,\alpha' \from G\to \Out$
	and either produces $\phi\in\Out$ such that $\phi(c) = c'$ and
	$(\alpha(g))^\phi = \alpha'(g)$ for all $g\in G$ or \texttt{No} if no
	such automorphism exists.
\end{theorem}

\subsection{Guirardel's core}\label{background:core}
Given two  simplicial  $T,  T'$ with $G$-actions, Guirardel \cite{guirardelcore}
constructed a \emph{core},  denoted $\Core(T, T')$, that captures the
compatibility of the two actions. Guirardel defines the core in a greater
generality than what we require, and we recall a specialized definition.

\begin{definition} 
	Suppose $T$ and $T'$ are simplicial trees with $\F$ action.  There
	exists a unique smallest non-empty, closed, connected, $\F$-invariant
	subset $C \subset T \times T'$ with convex fibers. We denote this set
	$\Core(T ,  T')$  and call it the \emph{core of }$T ,  T'$.
\end{definition}

The decision to use the connected subset is usually called the \emph{augmented
core}, but for our applications the connectivity is technically convenient, and
we pray the reader will accept our slight abuse of terminology. Guirardel
proved that a core always exists, and gave an explicit characterization of it
in terms of the dynamics of the $\F$ action on each tree~\cite{guirardelcore}.

We will use a different construction
of the Guirardel core found in different generalities in separate work of the
first two authors and in Behrstock, Bestvina, and
Clay~\cites{behrstockbestvinaclay,EdgarThesis,uniformsurgery}.  Given a
simplicial $\F$--tree $T$, let $\BP$ be a fixed vertex of $T$. Fix orientations
on all edges of $T$. For a given edge $e
\subseteq T$, the complement of the interior partitions $T$ into two connected
components, which we call $\delta^+(e)$ and $\delta^-(e)$ according to the
orientation of $e$. This induces a partition of the loxodromic elements
$\F_{hyp} = \partial^{+}(e)\cup
\partial^{-}(e)$ by the rule
\[ \partial^{\pm}(e) = \{ g\in \F | g^n\BP\in \delta^+(e) \text{ for large
$n$} \} \]

Note that the sets $\partial^{\pm}(e)$ are independent of the choice of $\BP$.  

\begin{definition}\label{Definition:Partition}
Consider an edge $a$ in $T$ and an edge $b$ in $T'$. We say $a \times b$ is an \emph{intersection square} 
if all of the following four intersections, as subsets of $\F$, are nonempty:
 \begin{align*}
 \partial^+(a) &\cap \partial^+(b) \neq \emptyset 
 &\partial^+(a) &\cap \partial^-(b)\neq \emptyset  \\
 \partial^-(a) &\cap   \partial^+(b)\neq \emptyset 
 &\partial^-(a) &\cap   \partial^-(b)\neq \emptyset 
 \end{align*}
	We call this intersection condition the \emph{4-sets condition}.
\end{definition}

A square $a\times b \subset \Core(T, T')$ if and only if it is an intersection
square~\citelist{\cite{EdgarThesis}*{Lemma 3.4.1}\cite{uniformsurgery}*{Lemma
3.4}}. 

\section{Geometric models of geometric EG strata} \label{section:GeometricModels}

One of the major ingredients in our algorithm is the relationship between the
irreducible pieces of Thurston normal form and the EG strata of
\emph{any} CT representing a geometric automorphism $\phi$. Before we make the
connection with Thurston normal form, we recall the definition of a geometric
model for an EG stratum, some key facts established by Handel and Mosher about
such models, and use them to derive invariants of $\phi$.

\subsection{Defining the models}

\begin{definition}[\cite{HandelMosher}*{Definition I.2.1}] A \emph{weak
geometric model}
for the EG stratum $H_r$ of a CT $f :  G \rightarrow G$ is the
	following collection of objects:
\begin{enumerate}

\item A compact, connected surface $S$ with nonempty boundary,
	enumerated as $\partial S = \partial_0 S\cup\cdots\cup\partial_{m} S$. The
	component $\partial_0 S$ is called the \emph{upper boundary} of $S$ and
	$\partial_{1}S,\ldots, \partial_{m}S$ are called the \emph{lower
	boundaries}.

\item For each lower boundary $\partial_{i}S, i = 1,...,m$, a homotopically
	nontrivial closed edge path $\alpha_i :  \partial_{i} S \rightarrow
	G_{r-1} $.

\item The 2-complex $Y$ that is the quotient of the disjoint union $S
	\sqcup G_{r-1}$ by the attaching map $\sqcup_i \alpha_i$ on the lower
	boundaries, with quotient map \[j \from S \sqcup G_{r-1} \rightarrow Y.\]

\item An embedding $G_r \hookrightarrow Y$ which extends the embedding $G_{r-1}
	\hookrightarrow Y$ such that $G_r \cap \partial_0 \Sigma$ is a single point
	denoted $p_{r}$, and there is a closed indivisible Nielsen path $\rho_r$ of
	height $r$ in $G_r$ and based at $p_{r}$, such that the loop $\partial_0
	\Sigma$ based at $p_{r}$ and the path $\rho_{r}$ are homotopic rel base
	point in $Y$. The existence of this embedding implies the existence of a
	deformation retraction $d:  Y\to G_r$ such that $d|_{\partial_0S}$ is a
	parameterization of $\rho_r$.

\item A homotopy equivalence $h :  Y \rightarrow Y$ and a homeomorphism \[g
	\from  (S,\partial_0 S) \to (S,\partial_0 S)\] whose mapping class
		$[g] \in \Map(S)$ is infinite order and irreducible, subject to the
	following compatibility conditions:
	\begin{enumerate}
	\item  The maps $(f | G_{r}) \circ d$ and $d \circ h$ are homotopic. 
	\item  The maps $j \circ g$ and $h \circ j$ are homotopic.
	\end{enumerate}
\end{enumerate}
\end{definition}

Following Handel and Mosher, we will make use of Thurston's theorem that $[g]\in
\Map(\Sigma)$ is infinite order and irreducible if and only if there is a
pseudo-Anosov representative, and when working with geometric models take $g$
to be pseudo-Anosov. The full data of a weak geometric model is quite
notationally heavy, for brevity we will say $Y$ is a weak geometric
model for the stratum $H_r$, and use the other notation introduced by the
definition in a standard fashion. When working with several geometric models we
will consistently decorate $Y$ and the other notation the same way, e.g. $S'$
is the surface associated to the weak geometric model $Y'$.

\begin{definition}
An EG stratum $H_r$ of a CT $f:  G\to G$ is \emph{geometric} if there is a
weak geometric model $Y$ for $H_r$.
\end{definition}

We recall that geometric EG strata are characterized by Nielsen paths.

\begin{fact}[\cite{HandelMosher}*{Fact I.2.3}]
\label{fact:nielsen-path-geometric}
For an EG stratum $H_r$ of a  CT $f:  G\to G$, the following are equivalent:
\begin{itemize}
\item $H_r$ is geometric.
\item There exists a closed, height $r$ indivisible Nielsen path $\rho_r$.
\item There exists a closed, height $r$ indivisible Nielsen path $\rho_r$ which
crosses each edge of $H_r$ exactly twice.
\end{itemize}
Furthermore each component of $G_{r-1}$ is non-contractible.
\end{fact}

Combining this fact with a lemma of Feighn and Handel gives two useful
structural results about geometric EG strata.

\begin{lemma}[\cite{FH:RecogThm}*{Lemma 4.24}]\label{maximal-subpaths-geom-eg}
Suppose $f:  G\to G$ is a CT and $H_r$ is a geometric EG stratum. If $E$ is an
edge of $H_r$ then each maximal subpath of $f(E)$ in $G_{r-1}$ is a Nielsen
path.
\end{lemma}

\begin{lemma}[\cite{FH:RecogThm}*{Lemma 4.24}]\label{no-geometric-enveloped-zero-strata}
Suppose $f:  G\to G$ is a CT and $H_r$ is a geometric EG stratum. Then $H_r$
does not envelop a zero stratum.
\end{lemma}

A weak geometric model $Y$ for a stratum $H_r$ only captures a CT $f:  G\to
G$ up to height $r$ of the filtration, but encapsulates most of the technical
notation and results. To connect the weak geometric model to the full
automorphism we recall the definition of a geometric model for a stratum.

\begin{definition}[\cite{HandelMosher}*{Definition I.2.4}]
A \emph{geometric model} $X$ for a geometric EG stratum $H_r$ of a CT $f: 
G\to G$ is the quotient of $Y \sqcup G$ obtained by identifying the embedded
copies of $G_r$ in each, where $Y$ is some weak geometric model for $H_r$, a
deformation retraction $d :  X \to G$ extending $d:  Y \to G$ and a
homotopy equivalence $h:  X \to X$ extending $h:  Y\to Y$ such that
$d\circ h$ is homotopic to $f\circ d$.
\end{definition}

The extra information contained in a geometric model but not a weak geometric
model needed for this work is encapsulated in the \emph{peripheral splitting}
associated to $X$. This is a $\mathcal{Z}$-splitting of $\F$ obtained from a
decomposition of the geometric model $X$ as a graph of
spaces~\cite{HandelMosher}*{Definition I.2.10}.  Vertex spaces
are the surface $S$ and components of the ``complementary subgraph''
$(G-H_r)\cup\partial_0S$; edge spaces are annuli (coming from boundary
components of $S$) and intervals (coming from attaching points of the $G-G_r$
to interior points of $H_r$). This splitting need not be minimal.
Using this splitting, one can see that for a geometric model $X$ of an EG
stratum of $f:  G\to G$, the composition $dj :  S\to G$ is
$\pi_1$-injective~\cite{HandelMosher}*{Lemma I.2.5}.

\subsection{Geometric models as invariants}
We need a few facts about the invariance of geometric models for EG strata that
are suggested by, but not proved in, Chapter I.2 of Handel and Mosher's
monograph~\cite{HandelMosher}. In that chapter, Handel and Mosher prove
the following structure proposition for the surface laminations of a geometric
stratum.

\begin{proposition}[\cite{HandelMosher}*{Proposition I.2.15}]\label{prop:geometric-model-laminations}
Given any finite type surface $S$ and any pseudo-Anosov homeomorphism
$g :  S\to S$ with stable and unstable lamination pair $\Lambda^{s}, \Lambda^{u} \subset  B(\pi_1 S)$, we have:
\begin{itemize}
\item $\Lambda^s/\Lambda^{u}$ is an attracting/repelling lamination pair for
	$g_\ast\in\out(\pi_1S)$. 
\item $\Lambda^s$ and $ \Lambda^u$ each fill $\pi_1 S$.
\end{itemize}
Furthermore, suppose that $f :  G \to G$ is a CT representative for some $\phi \in \Out$, that $H_r \subset  G$ is an EG-geometric stratum of $f$ corresponding to a geometric lamination pair
 $\Lambda^{\pm}$ of  $\phi$, and that the given surface $S$ is the surface
 associated to a weak geometric model for $f$ and $H_{r}$. If in addition $g$
	is the associated pseudo-Anosov homeomorphism of the weak geometric model then we have:
\begin{itemize}
\item The map $\widehat{dj} :  \mathcal{B}(\pi_1\Sigma) \to \mathcal{B}(\F)$ takes $\Lambda^u,\Lambda^s$ homeomorphically to $\Lambda^{+}, \Lambda^{-}$ respectively.
\item Every leaf of $\Lambda^{+}$ is dense in $\Lambda^{-}$, and similarly for $\Lambda^{-}$.
\item All leaves of $\Lambda^{+}$ and  $\Lambda^{-}$ are generic.
\item $\mathcal{F}_{supp}(\Lambda^{+}) = \mathcal{F}_{supp}(\Lambda^{-}) =
\mathcal{F}_{supp}([\pi_1 S])$.
\end{itemize}
\end{proposition}

From this structural result on laminations, they conclude that geometricity is
a property of a lamination pair $\Lambda^{\pm}\in \mathcal{L}^{\pm}(\phi)$
associated to an outer automorphism, independent of the geometric model or CT
used.

\begin{proposition}[\cite{HandelMosher}*{Proposition
I.2.18}]\label{prop:lamination-geometricity-invariant}
If $\Lambda^{\pm} \in \mathcal{L}^{\pm}(\phi)$ is an invariant lamination pair
associated to $\phi\in\Out$ and there is some rotationless iterate $\phi^k$ and
CT $f:  G\to G$ for $\phi^k$ such that $\Lambda^{\pm}$ is associated to a
geometric EG stratum of $f$, then for any CT $f':  G'\to G'$ of any
rotationless iterate $\phi^{m}$, the EG stratum of $f'$ associated to
$\Lambda^{\pm}$ is geometric.
\end{proposition}

In this section we pick up where Handel and Mosher leave off, and expand upon
\Cref{prop:lamination-geometricity-invariant}
to conclude that ``for a weak geometric model of a geometric EG stratum of a
rotationless automorphism, the homeomorphism type of the surface, the
mapping class of the pseudo-Anosov, and the conjugacy class of the inclusion of
the surface
fundamental group under $dj$ are invariants of $\phi$''; they do not depend on
the CT representative.

\begin{lemma}\label{lem:surfacegroupisaninvariant}
  Suppose $\phi \in \Out$ is a rotationless outer automorphism and
  $\Lambda^{\pm}\in\mathcal{L}^{\pm}(\phi)$ is a geometric lamination pair. Suppose $f:  G\to
  G$ and $f':  G'\to G'$ are two CT representatives for $\phi$.
  Let $H_r$ be the EG stratum of $f$ corresponding to $\Lambda$ with weak
  geometric model $Y$, and $H_s'$ be the EG stratum of $f'$ corresponding to
  $\Lambda$ with weak geometric model $Y'$. 
  Let $[K]\leq \F$ be the conjugacy class of the fundamental group of the surface $S$
  associated to $Y$; that is, $[K]=[(d\circ j)_*(\pi_1S)]$.  Similarly, let
  $[K']$ be the conjugacy class of the fundamental group of the surface $S'$ associated to $Y'$.  Then
  $[K]=[K']$.
\end{lemma}

\begin{proof}
  By \Cref{prop:geometric-model-laminations}, both $[K]$
  and $[K']$ carry $\Lambda^{\pm}$. Thus, for a representative $K\in [K]$ there
  is a subset $\widetilde{\Lambda}^{\pm}_K$ of 
  $\widetilde{\mathcal{B}}(K) \subseteq \widetilde{B}(\F)$ that is a lift of
  $\Lambda^{\pm}$. Let $K'\in [K']$ be the
  representative such that $\widetilde{\Lambda}^{\pm}_K\subseteq
  \widetilde{B}(K')\subseteq \widetilde{B}(\F)$. This choice exists since there
	a representative of $[K']$
	carries some lift of $\Lambda^\pm$ and all such lifts are $\F$
	conjugate.

  We first show that the intersection $\Gamma=K\cap K'$ has finite index in
  $K$ (and $K'$).  Indeed, for a contradiction suppose $[K: \Gamma]=\infty$ and consider the
  cover $S_\Gamma$ of $S$ corresponding to $\Gamma$.  Intersections of finitely
  generated subgroups of free groups are finitely generated, so $S_\Gamma$
  is an infinite sheeted cover with compact core.  The laminations
  $\Lambda^\pm$ are carried by both $K$ and $K'$, and therefore by $\Gamma$
  as well. Thus $\Lambda^\pm$ lift to $S_\Gamma$ and are therefore not filling.
	However, \Cref{prop:geometric-model-laminations} specifies that that $\Lambda^\pm$ are precisely the images of the
  stable and unstable laminations of the pseudo-Anosov model
  $g \colon S\to S$ under the map
  $\mathcal{B}(S)\hookrightarrow\mathcal{B}(\F)$. This is a contradiction: the stable and unstable
	laminations of any pseudo-Anosov homeomorphism are filling. Thus $\Gamma$ has
  finite index in $K$. By a symmetric argument $[K':\Gamma] < \infty$.

  It remains to show that $K'\leq K$; equality will follow by symmetry. Let $T$ be
  the $\F$-minimal subtree of the Bass-Serre tree coming from the peripheral
  splitting associated to the full geometric model $X$. The subgroup $K$
  is the stabilizer of a vertex $v$ in $T$, so it suffices to show
  that the $K'$-minimal subtree of $T$ is $v$.  Let $T_{K'}$ be the
  minimal subtree for $K'$.  If $T_{K'}$ had an infinite orbit of vertices, then
  $\Stab(v)\cap K'=K\cap K'=\Gamma$ would have infinite index in $K'$, which we
  have ruled out. So $T_{K'}$ has finitely many vertices and contains $v$. Since
  $T_{K'}$ is finite, the fixator $\Fix(T_{K'}) \le K'$ is finite-index in $K'$
  and contained in each edge stabilizer of $T_{K'}$. Recall now that $T$ is a
  subtree of the Bass-Serre tree of a $\mathcal{Z}$-splitting, hence all edge
  groups are trivial or cyclic. However, the group $K'$ is the fundamental
  group of a surface carrying a pseudo-Anosov homeomorphism, so is a free group
	of rank at least 2. Since $\Fix(T_{K'})$ is finite index in $K'$ it is
	also free of rank at least 2, and contained in each edge stabilizer of
	$T_{K'}$. The only way for this to be possible is if $T_{K'}$ contains
	no edges, that is $T_{K'}=v$. Hence $K'\leq K$. As equality follows by
	symmetry this completes the proof.
\end{proof}

\begin{lemma}\label{lem:boundaryclassesareinvariant}
  Suppose $f:  G\to G$ is a CT representative of a rotationless $\phi \in
  \Out$ and $\Lambda^{\pm}\in \mathcal{L}^\pm(\phi)$ is a geometric lamination
  pair. Let $H_r$ be the geometric EG stratum of $f$ associated to
  $\Lambda^{\pm}$ and $Y$ a geometric model for $H_r$.
  Let $[K]\leq \F$ be the fundamental group of the surface $S$
	associated to $Y$; that is, $[K]=[(d\circ j)_*(\pi_1S)]$.
  For any
  representative $K$ the set of $K$-conjugacy classes representing
  $\partial S$ is independent of the CT and depends only on $\phi$.
\end{lemma}
\begin{proof}
  Let $g$ be the pseudo-Anosov homeomorphism of $S$ coming from the
  geometric model. Denote the stable and unstable laminations of $g$ by
  $\Lambda^s$ and $\Lambda^u$. By
  \Cref{prop:geometric-model-laminations} $\widehat{dj}$ is a homeomorphism
  from $\Lambda^u$ to $\Lambda^+$ and from $\Lambda^s$ to $\Lambda^-$. We will
  identify the boundary classes of $K$ using principal lifts, so some care with
  basepoints is required.

  Fix the basepoint $b\in H_r \subseteq G$, as well as a basepoint $\ast\in
  S$ such that $dj(\ast) = b$ in the image of the geometric model attaching
  maps. Fix basepoints $\tilde{b} \in \tilde{G}$, and $\tilde{\ast}\in
  \tilde{S}$,  and a lift $\widetilde{dj} :  \tilde{S}\to \tilde{G}$
  covering $dj$ at these base-points. Let $K\in [K]$ be the representative image
  of $\pi_1(S)$ picked by these base0point choices.

  For each principal lift $\Psi$ of $\phi$ that fixes some element of $[K]$
  there is an iso-gredient principal lift $\Phi$ fixing $K$. Since $[K]$ carries
  $\Lambda^{\pm}$ by \Cref{prop:geometric-model-laminations}, there
  is a lift of
  the laminations $\Lambda^{\pm}$ to $\tilde{\Lambda}^\pm \subseteq
  \tilde{\mathcal{B}}(K)$. A principal lift
  $\Phi$ satisfies $|\tilde{\Lambda}^+ \cap \Fix_N(\Phi)| \ge 3$ if and only if
  it fixes $K$.

  Let $\Phi$ be a principal lift of $\phi$ fixing $K$, and select a realization
  $\tilde{f}:  \tilde{G}\to \tilde{G}$ so that $\tilde{f}(b) \in
  \widetilde{dj}(\tilde{S})$. This is possible since $\Phi$ fixes $K$.
  Let $\tilde{g}$ be a lift of $g$ such that
  $\widetilde{dj}(\tilde{g}(\tilde{\ast})) = \tilde{f}(b)$.
  By the homotopy lifting property of covers and item 5a in the definition of
  geometric model, the action of $\widetilde{g}$ on the lifts of $\Lambda^s$ and
  $\Lambda^u$ is conjugate via $\widetilde{dj}$ to the action of $\widetilde{f}$ on
  $\tilde{\Lambda}^{\pm}$. Thus $\tilde{g}$ is a principal lift of $g$.

  In fact, every principal lift of $g$ arises this way. Indeed, if $\tilde{g}$
  is a principal lift of $g$, then, again by the homotopy lifting property and
  item 5a in the definition of geometric model, the lift of $f$ such that
  $\tilde{f}(b) = \widetilde{dj}(\tilde{g}(\ast))$ is a realization of a
  principal lift of $\phi$.

  By \Cref{lem:surfacegroupisaninvariant} the conjugacy class $[K]$ is an
  invariant of $\phi$. For each representative $K$ the boundary conjugacy
  classes of $S$ are determined by the dynamics of the principal lifts of $g$
  on
	$\widetilde{\mathcal{B}}(K)$, see \Cref{prelim-principal-region-structure-theory}.
  Since each principal lift of $g$ fixing $K$ arises as a
  principal lift of $\phi$ fixing $K$, these dynamics are invariants of $\phi$.
  This completes the proof.
\end{proof}

\begin{corollary}\label{cor:godgivensurfacepiece}
	Suppose $\Lambda^\pm \in \mathcal{L}(\phi)$ is a geometric lamination pair.
	Let $f:  G\to G$ be a CT representing $\phi$ and suppose $H_r$ is the
	geometric EG stratum of $f$ corresponding to $\Lambda^{\pm}$.
	The surface $S$ associated to a weak geometric model $Y$ for $H_r$ and the
	mapping class of $g:  S\to S$ are both invariants of the outer automorphism
	$\phi$ and do not depend on $Y$ or $f$.
\end{corollary}

\begin{proof}
	Suppose $f':  G\to G$ is another CT for $\phi$ and $H_s'$ is the
	geometric EG stratum of $f'$ corresponding to $\Lambda^{\pm}$. Let $S'$
	be the surface associated to a geometric model $Y'$ of $H_s'$. We will show
	that there is a homeomorphism $\eta:  S\to S'$ such that $[\eta g
	\eta^{-1}]= [g']\in \Map(S')$.

	Pick base-points in the strata $H_r$ and the surface $S$ to fix a
	representative $K = dj_\ast (\pi_1(S),\ast) \le \F$. By
	\Cref{lem:surfacegroupisaninvariant}, $[K] = [d'j'_\ast(\pi_1(S'))]$,
	so there is a choice of basepoint in $H_s'$ and on $S'$ such that
	$d'j'_\ast(\pi_1(S',\ast') = K$.

	Let \[F_\ast = ({d'j'}_\ast^{-1})
	\circ dj_\ast :  \pi_1(S,\ast) \to \pi_1(S',\ast) \] be the induced
	isomorphism (this is indeed well defined as both $dj_\ast$ and $d'j'_\ast$
	are isomorphisms onto their images). Since $\pi_1(S,\ast)$ is free, there
	is a homotopy equivalence $\Theta$ inducing this isomorphism.

	By \Cref{lem:boundaryclassesareinvariant}, $\Theta_\ast$ takes the
	boundary classes of $S$ to the boundary classes of $S'$, so by the
	Dehn-Nielsen-Baer theorem (\Cref{dehn-nielsen-baer}) $\Theta$ is
	homotopic rel boundary to a homeomorphism $\eta:  S\to S'$.

	Finally, by \Cref{prop:geometric-model-laminations},
	$\widehat{\eta} = \widehat{\Theta}$ induces homeomorphisms
	from $\Lambda^s$ to ${\Lambda^s}'$ and $\Lambda^u$ to ${\Lambda^u}'$, 
	and $g$ and $g'$ have the same dilatation, thus
	\[ [\eta g \eta^{-1}] = [g'] \in \Map(S) \]
	by the recognition theorem for pseudo-Anosov mapping classes~\cite{McCarthy}*{Theorem 1}.
\end{proof}

\section{Deciding geometricity for rotationless automorphisms}\label{sec:rotationless}

In this section we introduce \Cref{alg:rotationless}, which decides if
a rotationless outer automorphism $\phi$ is geometric, and verify its
correctness. Our description of \Cref{alg:rotationless} begins with an
investigation of the properties of CTs representing a geometric
automorphism $\phi$. 

This investigation begins with a particular CT coming from the surface on which
$\phi$ is realized.
We then examine the extent to which these properties are shared by all
CTs representing $\phi$. In this way, we provide a list of computable necessary
conditions for any CT representing a rotationless automorphism to be
geometric. These conditions are split into dynamical constraints on the strata
of a representative CT (\Cref{sec:geometricCT}) and algebraic
constraints on the fixed subgroups (\Cref{sec:fixedsubgroup}).

Feighn and Handel~\cite{FH:CTs} provide an algorithm to
compute a CT representative of a rotationless outer automorphism. We arrive at
\Cref{alg:rotationless} by computing a surface
and a realizing homeomorphism from a CT satisfying the necessary conditions, demonstrating their sufficiency.

\subsection{CT representatives of a geometric automorphism}\label{sec:geometricCT}

The following notation will be used throughout the
remainder of this subsection and the next.

\begin{notation}\label{NotationPhi}
  Fix a rotationless element $\phi\in\Out$ which we assume to be geometric.
	Let $\Sigma$ be a (possibly non-orientable) connected compact surface 
	whose fundamental group
	is identified with $\F$ ($\pi_1\Sigma\simeq \F$).  Let
	$g\colon\Sigma\to\Sigma$ be a homeomorphism inducing $\phi$
	($[g_*]=\phi$). We assume further that $g$ is in Thurston normal form
	for its homotopy class, so that it has no periodic behavior as detailed in \Cref{NoPeriodicBehavior}.
\end{notation}

\begin{lemma}\label{NoPeriodicBehavior}
	The Thurston normal form of $g$ is without periodic behavior. That
  is, let $\{c_1,\ldots,c_m\}$ be the canonical reduction system of
  $g$ and choose representatives $c_i$ with pairwise disjoint closed
  neighborhoods $R_1,\ldots,R_m$.  Let $R_{m+1},\ldots, R_{m+n}$
  denote the closures of the connected components of the complement of the
	reduction system
  $\Sigma-\cup_{i=1}^mR_i$.  Let
  $\eta_i\colon \Homeo(R_i,\partial R_i)\to \Homeo(\Sigma)$ denote the homomorphism
  induced by the inclusion $R_i\hookrightarrow \Sigma$.  Then $g(R_i)=R_i$ for
  all $i$ and moreover,
  \begin{equation*}
    g=\prod_{i=1}^{m+n}\eta_i(g_i)
  \end{equation*}
	where $[g_i]\in\Map(R_i)$ is a power of a Dehn twist for
	$i\in\{1,\ldots,m\}$ and $[g_i]\in\Map(R_i)$ is either pseudo-Anosov
  or the identity for $i\in\{m+1,\ldots,m+n\}$.
\end{lemma}

\begin{proof}
  This follows immediately from the fact that if
  $\phi\in\Out$ is rotationless and $c$ is a periodic conjugacy
  class, then $c$ is in fact fixed by $\phi$.  Were it necessary to
  pass to a power of $g$ to obtain the above form, there would be a
  conjugacy class which is periodic but not fixed. 
\end{proof}

We now choose a nested sequence
$\emptyset=\mathcal{F}_0\leq \mathcal{F}_1\leq\ldots
\mathcal{F}_M=\F$ of $\phi$-invariant free factor systems determined
by (the fundamental groups of) an appropriately chosen maximal nested
sequence $Q_1\subset Q_2\subset\ldots\subset Q_M$ of $g$-invariant
subsurfaces of $\Sigma$.  One way to ensure such a sequence of
subsurfaces determines a nested sequence of free factor system is as
follows.  Using the notation in \Cref{NoPeriodicBehavior}, assume
after reordering that $R_1,\ldots, R_b$ are annular neighborhoods of
the boundary components of $\Sigma$.  Start by defining
$Q_i=R_1\cup \ldots \cup R_i$ for $i\in\{1,\ldots,b-1\}$; note that we
have omitted the last boundary component, as carried by $R_b$.  Then ``work towards'' the final
boundary component as follows: define a partial order on the remaining
invariant subsurfaces $R_i$ for $i\in \{b+1,m+n\}$ by $R_i\leq R_j$ if
$\mathcal{F}_{supp}(\pi_1R_i,\pi_1Q_{b-1})
\sqsubset \mathcal{F}_{supp}(\pi_1R_j,\pi_1Q_{b-1})$.  Then for
$i\in\{b+1,\ldots,m+n\}$ inductively define $Q_i=Q_{i-1}\cup R$ where
$R$ is any subsurface which is minimal with respect to this partial
order.  Finally, define $\mathcal{F}_i=\mathcal{F}_{supp}(\pi_1Q_i)$ as
the smallest free factor system carrying the fundamental group(s) of
the not necessarily connected subsurface.

\begin{notation}\label{NotationCT}
  Let $f\colon G\to G$ be a CT representing $\phi$ with filtration
  \[\emptyset=G_0\subset \ldots\subset G_N=G\] and realizing the (not
  necessarily maximal) nested sequence of $\phi$-invariant free factor
  systems defined above.
\end{notation}

The first step in the construction of a CT involves completing the
prescribed nested sequence of $\phi$-invariant free factor systems to
a maximal one.  We will distinguish those core subgraphs $G_r$ whose
fundamental groups are elements of the prescribed filtration (that is
$\mathcal{F}_{supp}([\pi_1(G_r)]) \in\{\mathcal{F}_i\}$) from those
core subgraphs whose fundamental groups are not in the prescribed
filtration; will say subgraphs in the former category and their strata
\emph{determine subsurfaces}, while subgraphs in the latter category
\emph{do not determine subsurfaces}.

\subsubsection{Strata of the preferred CT}
In the previous section we proved that the surface data of a geometric EG
stratum of a CT is an invariant. We now use $f:  G\to G$ to show
that these invariants agree with the pieces of the Thurston normal form of $g$.

\begin{observation}\label{obs:pAFreeFactors}
  Suppose the restriction of $g\colon \Sigma\to\Sigma$ to a subsurface
  $R$ is pseudo-Anosov, and that $Q_i$ is the subsurface of $\Sigma$
  such that $Q_i=Q_{i-1}\cup R$ as defined above, with associated
  $\phi$-invariant free factor system $\mathcal{F}_i$.  Then there are
  no $\phi$-invariant free factors properly containing
  $\mathcal{F}_{i-1}$ and properly contained in $\mathcal{F}_i$.
\end{observation}

\begin{lemma}
  \label{EGStrataAreGeometric}
  If $H_r$ is an EG stratum of $f\colon G\to G$, then $H_r$ is
  geometric and the subsurface $R$ of $\Sigma$ such that $Q_r = Q_{r-1} \cup R$
  can be attached to $G_{r-1}$ to produce a weak geometric model for $H_r$.
\end{lemma}
\begin{proof}
  Since $H_r$ is an EG stratum, $G_r$ is necessarily a core graph; let
  $\mathcal{F}_{supp}([\pi_1G_r])$ be the associated $\phi$-invariant
  free factor system.  Again because $H_r$ is EG there are conjugacy
  classes in $\mathcal{F}_{supp}([\pi_1G_r])$ that grow exponentially
  under iteration by $\phi$ and are not contained in
  $\mathcal{F}_{supp}([\pi_1G_{r-1}])$.  In particular, this implies
  that the $\phi$-invariant free factor system
  $\mathcal{F}_{supp}([\pi_1G_r])$ determines a subsurface $Q_i$ of
  $\Sigma$. Let $R$ be
  the subsurface of $\Sigma$ such that $Q_i=Q_{i-1}\cup R$. Note that the presence of exponentially growing conjugacy classes in $R$ and the maximality of the filtration imply that $g|_R$ is pseudo-Anosov. There is
  a simple closed curve $c$ in $\partial R$ which is not homotopic into
  $Q_{i-1}$ and is fixed by $g$ (it's the ``upper boundary
  component'').  Since $g|_R$ is pseudo-Anosov, 
  \Cref{obs:pAFreeFactors} implies that the core graph of $G_{r-1}$ 
  has fundamental group equal to $\mathcal{F}_{i-1}$.
  Abusing notation and thinking of $c$ as an element of $\F$, it is
  clear that $[c]$ is not carried by $\mathcal{F}_{i-1}$ and therefore
  not carried by $G_{r-1}$.  Thus there is a conjugacy class in $G$ of
  height $r$ that is $\phi$-invariant, so is represented by a height $r$ closed
  indivisible Nielsen path $\rho_r$. This is equivalent to $H_r$
  being geometric by \Cref{fact:nielsen-path-geometric}.

  Now suppose $S\stackrel{j}{\to} X\stackrel{d}{\to} G$ is a geometric model
  for $H_r$. The argument is
  similar to \Cref{lem:surfacegroupisaninvariant}. Both $\pi_1(R)$ and $\pi_1(S)$ carry
  the laminations of the stratum $\Lambda^{\pm}$, so the intersection $K =
  \pi_1(R) \cap \pi_1(S)$ is finite index in both groups. Indeed, in both
  surfaces $\Lambda^{\pm}$ are realized as attracting laminations
  for pseudo-Anosov homeomorphisms, so each is filling and carried by $K$; this
  is impossible if $K$ is infinite index.

  Let $X$ be a geometric model of $H_r$ and $T$ be the minimal
  tree associated to the corresponding $\mathcal{Z}$-splitting. As in
  \Cref{lem:surfacegroupisaninvariant} we conclude that $\pi_1(R)\le \pi_1(S)$. While
  we cannot appeal to symmetry exactly, to see that $\pi_1(S) \le \pi_1(R)$,
  consider the Bass-Serre tree $T_R$ coming from the $\mathcal{Z}$-splitting of
  $\Sigma$ induced by $\partial R$. The minimal tree $T_S$ for the action of
  $\pi_1(S)$ must be finite since the intersection $K$ is finite-index in
  $\pi_1(S)$; since $\pi_1(S)$ is free of rank at least 2 this implies it is a
  single vertex.

  Finally the upper boundary $\partial_0S$ represents $[c]$ by the definition
  	  of a geometric model so upper boundaries of $S$ and $R$ are identified.
  	  Since $\pi_1(S) \cong \pi_1(R)$
  to see that the lower boundaries are identified we appeal to
  \Cref{lem:boundaryclassesareinvariant}. The boundary classes are an
  invariant of the lamination pair $\Lambda^{\pm}$, independent of the CT or
  geometric model. Moreover, as the surface and free group notions of principal
  lift coincide~\cite{FH:RecogThm}, the lamination pair $\Lambda^{\pm}$
  determines the boundary classes of $R$. Thus, for an appropriate choice of
  basepoints we obtain an isomorphism $\Theta_\ast = \iota_\ast^{-1}\circ dj_\ast : 
  \pi_1(S) \to \pi_1(R)$ that takes boundary classes to boundary classes,
  respecting the distinction of upper and lower classes. As in the proof of
  \Cref{cor:godgivensurfacepiece} it follows from the
  Dehn-Nielsen-Baer theorem that there is a homeomorphism $\eta :  S\to R$
  conjugating the
  geometric model homeomorphism to the restriction of the Thurston normal form
  to $R$. The homeomorphism $\eta$ can be used to attach $R$ to $G_r$ to
  produce a new weak geometric model.
\end{proof}

Geometricity also imposes a strict constraint on NEG strata of $f$, capturing
the growth dichotomy for conjugacy classes under iteration by a surface
homeomorphism.

\begin{lemma}\label{NEGStrataAreLinear}
  If $H_r$ is a non-fixed NEG stratum, then it is linear.
\end{lemma}
\begin{proof}
  Suppose for a contradiction that there exists a nonlinear NEG
  stratum and let $H_r$ be the lowest such.  The stratum $H_r$ 
  consists of a unique edge $E$ such that $f(E)=Ew$ for some
  completely split conjugacy class $w$.  We first show that if $w$
  contains an EG edge $E'$ as a splitting unit, there will be a
  conjugacy class $[\sigma]$ in $\F$ whose asymptotic growth rate is
  super-exponential and this behavior does not occur for geometric
  automorphisms.

  To produce such a conjugacy class, we simply need to
  construct a completely split conjugacy class containing $E$ as a
  splitting unit. To do so, consider the smallest integer $s \ge r$
  $r$ such that $G_s$ is a core graph.  

  We first argue that $H_{s}$ cannot be an EG stratum.  If $H_{s}$ were EG,
  then it would be geometric by \Cref{EGStrataAreGeometric}. Thus, by
  \Cref{maximal-subpaths-geom-eg}, for each edge $E''$ of $H_s$, the
  maximal subpaths of $f(E'')$ in $G_{s-1}$ are Nielsen paths. 
  Thus, the EG case of \Cref{moving-up-through-the-filtration} implies the
  existence of a Nielsen path $w$ of height $r$ whose first edge is $E$. This
  is a contradiction to (NEG Nielsen Paths), since $E$ is non-linear. Thus,
  $H_{s}$ cannot be EG.

  So we conclude that $H_{s}$ is NEG and the relation between $H_s$ and $G_r$
  is described in the NEG case of \Cref{moving-up-through-the-filtration}.
  Since $G_s$ is core, there exists a closed loop $\sigma$ in $G_s$
  crossing $E$ either once or twice according to whether or not $E$ is
  separating in $G_s$. Iterating $f$, we may assume that $\sigma$ is
  completely split, and we claim $E$ is a splitting unit in the complete
	splitting of $\sigma$. Indeed, the edge $E$ is not contained in a zero stratum, so it is not
	a taken path. It also cannot appear in an indivisible Nielsen path or
	exceptional path in the complete splitting of $\sigma$: Gupta and
	Wigglesworth characterize the structure of
  indivisible Nielsen paths and exceptional paths that cross non-fixed
	irreducible strata \cite{GW:Loxodromics}*{Lemma 7.4}, and their
	characterization rules out the possibility of $E$ appearing in either
	type of splitting unit. 

	Therefore, the asymptotic growth of iterates of $E$ gives a lower bound for
	the asymptotic growth of $\sigma$. Our assumption that $f(E)$ contains an EG edge as
	a splitting unit implies that $\ell(f^n(E))\ge n\lambda^n$ (here
	$\lambda$ is the exponential growth rate of the EG edge $E'$). Hence $\ell(f^n(\sigma)\ge n\lambda^n$ but this 
	cannot occur for geometric automorphisms. Thus $w$ contains no EG edges as splitting units.

  The minimality of $H_r$ implies that, absent EG splitting units, $w$ must
  contain a linear edge $E'$ in its complete splitting. A similar argument
  then implies the existence of a conjugacy class that grows quadratically
  under iteration by $\phi$ (and hence by $g$), something that cannot happen
  for conjugacy classes under iteration by a surface homeomorphism.
\end{proof}

\subsubsection{Strata of an arbitrary CT}
We now consider an arbitrary CT representing the geometric
automorphism $\phi$.

\begin{notation}\label{Notationf}
  For the remainder of this section, we let $f'\colon G'\to G'$ be
  \emph{any} CT representing the geometric automorphism $\phi$.  In
  particular, $f'$ need not have anything to do with a surface
  $\Sigma$ on which $\phi$ can be realized.
\end{notation}

\begin{lemma}\label{EGAndNEGStrataOff'}
  If $f'\colon G'\to G'$ is as above, then every EG stratum of $f'$ is
  geometric and every nonfixed NEG stratum is linear.
\end{lemma}

\begin{proof}
  The first statement is implied by \Cref{EGStrataAreGeometric}
  together with \Cref{prop:lamination-geometricity-invariant}. For
  the second, we simply note that the proof of \Cref{NEGStrataAreLinear}
  relies only on the asymptotic growth rates of conjugacy classes in $\F$,
  and this depends only on $\phi$ and not $f$.
\end{proof}

As a consequence of \Cref{EGAndNEGStrataOff'} we obtain two
more necessary conditions for a rotationless automorphism to be geometric:

\begin{corollary}\label{cor:nozeroornest}
  There are no elements $\Lambda_1,\Lambda_2\in\mathcal{L}(\phi)$ such
  that $\Lambda_1\supsetneq \Lambda_2$.
  Any CT $f'\colon G'\to G'$ representing $\phi$ does not
  have a zero stratum.
\end{corollary}
\begin{proof}
   Since $\phi$ is geometric, by \Cref{EGAndNEGStrataOff'} every EG
   stratum $H'_r$ is geometric. 

   For every $E$ edge in an EG stratum $H'_r$
   \Cref{maximal-subpaths-geom-eg} implies each maximal subpath of $f'(E)$
   in $G'_{r-1}$ itself a Nielsen path. In particular, there are no EG edges of
   height $<r$ in the complete splitting of $f'(E)$, so the lamination
   $\Lambda$ associated to $H'_r$ cannot contain a sub-lamination.

  Every zero stratum is enveloped by some EG stratum, but
  \Cref{no-geometric-enveloped-zero-strata} states that geometric EG
  strata do not envelop zero strata.  
\end{proof}

In fact, the pseudo-Anosov pieces used for geometric models are an invariant of
$\phi$ independent of the CT.

\begin{lemma}\label{pASubsurfaces=EGStrata}
	With the conventions established by \Cref{NotationPhi}, \Cref{NotationCT}, and \Cref{Notationf} above, there is a bijection
  \begin{equation*}
    \{\text{EG strata of }f'\}\longleftrightarrow
    \{R\subseteq\Sigma\mid g(R)=R\text{ and }g|_R\text{ is pA}\}.
  \end{equation*}
  Moreover, a subsurface on the right-hand side of the bijection can be used
  as the surface part of a geometric model for the corresponding stratum.
\end{lemma}
\begin{proof}
  For the CT $f\colon G\to G$ obtained using a
  filtration coming from the surface $\Sigma$; this is a consequence
  of \Cref{EGStrataAreGeometric}. For
  any other CT $f'$ the requisite the bijection comes
  from the bijection between EG strata of $f'$ and
  attracting laminations $\mathcal{L}(\phi)$. That this bijection provides
  surface pieces of geometric models follows from
  \Cref{cor:godgivensurfacepiece}.
\end{proof}

In summary, we conclude that for any CT $f'$ representing a rotationless
geometric automorphism $\phi$, every stratum of $f'$ must be either fixed,
linearly growing, or geometric.

\subsection{The fixed subgroups}\label{sec:fixedsubgroup}

Beyond the ``stratum constraints'' from the previous section, the fixed
subgroups of a rotationless geometric automorphism reflect the fixed
subsurfaces of the Thurston normal form. The precise formulation is the notion
of a $\partial$-realizable set of conjugacy classes, 
\Cref{def:boundary-realizable}. We first show that this property is algorithmic,
and then apply it to the fixed subgroups of a rotationless
automorphism.

We continue to use \Cref{NotationPhi}, \Cref{NotationCT}, and \Cref{Notationf} in this subsection.

\subsubsection{$\partial$-realizable sets in a free group}
\begin{definition}\label{def:boundary-realizable}
  Let $\F$ be a finite rank free group and $\mathcal{C}$ a finite set of
  conjugacy classes of $\F$.  We say that $\mathcal{C}$ is
  \emph{$\partial$-realizable} if there exists a surface $\Sigma$ and
  an identification $\F\simeq \pi_1(\Sigma)$ such that the free homotopy class corresponding to every element
  of $\mathcal{C}$ is the class of a boundary component of $\Sigma$.  We do not allow a class
  $[c]\in \mathcal{C}$ to be a proper power of a boundary component. A finite
  multiset $\mathcal{C}$ is $\partial$-realizable if either $\mathcal{C}$ is in
  fact a $\partial$-realizable set, or if $\F = \langle c\rangle$ and
  $\mathcal{C}$ is $[c]$ with multiplicity at most 2.
\end{definition}

Suppose that $\mathcal{C}$ is $\partial$-realizable and let $\Sigma$ be a
surface witnessing this fact.  If $[\partial \Sigma]$ is the collection of conjugacy
classes of elements of $\pi_1(\Sigma)\simeq\mathbb{F}$ determined by
the boundary components of $\Sigma$, then there are two possibilities:
either $\mathcal{C} = [\partial \Sigma]$ or else $\mathcal{C}\subsetneq [\partial
\Sigma]$.  In the latter case, $\mathcal{C}$ will be
part of a basis for $\mathbb{F}$, so Whitehead's algorithm~(\Cref{whitehead}) will determine if
such an $\mathcal{C}$ is $\partial$-realizable.  The following corollary shows
that the same holds when $\mathcal{C}= [\partial \Sigma]$ again using Whitehead's algorithm.

\begin{corollary}\label{cor:partial-realizable-algo}
	There is an algorithm (\textsc{$\partial$-realizable?}) that
	determines whether a multiset $\mathcal{C}$ of conjugacy classes is
	$\partial$-realizable in $\mathbb{F}$.
\end{corollary}
\begin{proof}
By definition a multiset is $\partial$-realizable if either it is a
$\partial$-realizable set (not multiset) or if $\F \cong\mathbb{Z}$ and $\mathcal{C}$ is a generator of
$\mathbb{F}$ with multiplicity at most 2. This second condition is readily
computed, so it remains to give an algorithm for the case
that $\mathcal{C}$ is a set.

	Observe that a \emph{set} of conjugacy classes $\mathcal{C}$ is
$\partial$-realizable if and only if $\mathcal{C}$ is in the $\Out$ orbit of
the set of conjugacy classes $[\partial S_{n,i}]$ represented by the boundary
	of some standard surface $S_{n,i}$. Thus,
	$\partial$-realizability is computable with Whitehead's algorithm~(\Cref{whitehead}): For each
	standard surface $S_{n,i}$ of the relevant rank, use \Cref{whitehead}  to determine if any
	$|\mathcal{C}|$-subset of $[\partial S_{n,i}]$ and
  $\mathcal{C}$ are in the same $\Out$-orbit.
\end{proof}

\subsubsection{$\partial$-realizable sets and the fixed subgroups of $\phi$}
We now consider the collection of conjugacy classes of
subgroups $\Fix(\phi)$ for the geometric automorphism under
consideration. Indeed, computing $\Fix(\phi)=\{[K_1],\ldots,[K_l]\}$
from the CT $f\colon G\to G$ can be done easily~ \cite{FH:CTs}*{\S 8}.  On the
other hand, computing $\Fix(\phi)$ from $g\colon\Sigma\to\Sigma$ is also
straightforward from its Thurston normal form: $\Fix(\phi)$ consists of the
conjugacy classes of fundamental groups of subsurfaces on which $g$ restricts
to the identity. This second characterization leads to the conclusion developed
in this section: that the boundary curves of the fixed subsurfaces can be
computed from a CT, and geometricity on the identity components reduces to a
question of $\partial$-realizability.

We recall the details of the computation of $\Fix(\phi)$ from a CT
$f\colon G\to G$, as a familiarity with this procedure is important in
the sequel; the reader is directed to Feighn and Handel~\cite{FH:CTs}*{\S 9-11}
for complete details of the construction and its computability. The reader may
wish to look ahead to \Cref{ex:SurfaceExample}.

\begin{definition}
Define the graph $\hat{S}(f)$ as follows.  Start with the subgraph $\hat{S_1}(f)$ of $G$ consisting of
all vertices in $\Fix(f)$ and all fixed edges.  Given a linear edge
$E$ of $G$, we have $f(E)=E u_E^d$ for some $d\neq 0$ and some
root-free loop $u_E$ in $G$ that is fixed by $f$ (up to free
homotopy). For each such edge, attach a ``lollipop'' $Y_E$ to
$\hat{S}_1(f)$; $Y_E$ is the union of an edge labeled $E$ and a circle
labeled $u_E$, which is attached at the initial vertex of $E$,
considered as an element of $\hat{S}_1(f)$. For each geometric EG stratum $H_r$
with an indivisible Nielsen path of height $r$, choose one such
	indivisible Nielsen path $\rho$ (there are only two and they differ by
	a choice of orientation) and attach an edge path labeled by $\rho$ to
	$\hat{S}_1(f)$ with endpoints equal to those of $\rho$. The result is
	$\hat{S}(f)$. 
\end{definition}

When $v\in\Fix(f)$, we abuse notation and also denote by $v$ the
unique vertex of $\hat{S}_1(f)\subset\hat{S}(f)$ labeled by $v$.  The
component of $\hat{S}(f)$ containing this vertex is denoted by
$\hat{S}(f, v)$. Each component of $\hat{S}(f)$ has its fundamental
group identified with a subgroup of $\F$ by the immersion determined by edge labels. The collection of conjugacy classes
of such subgroups is precisely $\Fix(\phi)$. The graph $\hat{S}(f)$ is clearly
computable from the data of a CT and we will use its components as computable
representatives of $\Fix(\phi)$.

Heuristically, linear strata in a CT representing a geometric outer
automorphism should correspond to Dehn twists in the associated
surface. This motivates the following definition, which will record
the set of Dehn twist curves in a reducing system for such an
automorphism. This definition should be compared to Handel and Mosher's
definition of $\text{Twist}(\phi)$~\cite{HandelMosher}*{Definition II.2.7}.
Taken as a set $\mathcal{L}_K$ is the intersection up to conjugacy of $K$ and
 $\text{Twist}(\phi)$, however the multiset structure is necessary for
 determining geometricity.

\begin{definition}\label{DefLK}
	For each $[K]\in\Fix(\phi)$, we define a multiset $\mathcal{L}_K$ of
  conjugacy classes of $K$ as follows ($\mathcal{L}$ is for ``linear'').
Consider each linear edge
  $E$  of $f\colon G\to G$ in turn. For each edge, write \[ f(E)=Eu_E^d\] for
  some root-free reduced loop $u_E$ in $G$ and some integer $d\neq 0$.
  Let $v$ and $v'$ be the initial and terminal endpoints of $E$
  respectively, both of which are necessarily in $\Fix(f)$. If $v$ is in
	the component of $\hat{S}(f)$ corresponding to $[K]$ add $Eu_E\overline{E}$
  to $\mathcal{L}_K$. If $v'$ is in the component of $\hat{S}(f)$ corresponding
	to $[K]$ add $u_E$ to $\mathcal{L}_K$.
\end{definition}

Our definition of $\mathcal{L}_K$ is computable from a CT representative, for
it to be useful we also must know that $\mathcal{L}_K$ is an invariant of
$\phi$ independent of the choice of CT. To do this we will make use of an
alternate definition, in terms of the \emph{axes} of $\phi$. Recall that a
root-free conjugacy class $\mu$ is an \emph{axis of $\phi$} if there are
distinct principal lifts $\Phi, \Psi\in P(\phi)$ that fix a representative $u$
of $\mu$. The number of distinct principal lifts in $P(\phi)$ fixing $u$ is the
\emph{multiplicity} of $u$, denoted $m(\mu)$. This does not depend on the particular
representative, $\mu$ and its multiplicity are invariants of $\phi$. For a
fixed $u$ representing $\mu$ there is
 a unique \emph{base lift} $\Phi_0 \in P(\phi)$, characterized
as corresponding to the unique lift $\tilde{f}_0$ with fixed points in the axis
of $u$ and commuting with the covering translation
$\tau_u$~\cite{FH:RecogThm}*{pg. 95}. With this notion of base lift, Feighn
and Handel connect the axes of an outer automorphism $\phi$ to the linear edges
of a CT representative.

\begin{lemma}[\cite{FH:RecogThm}*{Lemma
4.40}]\label{axes-determine-linear-edges}
Suppose that $\phi$ is forward rotationless and that the unoriented conjugacy
class $\mu$ is an axis for $\phi$. Let $f :  G\to G$ be a CT representative
of $\phi$. Fix a representative circuit $u$ in $G$ for $\mu$ and let
$\Phi_0$ be the base lift corresponding to the choices of $u$ and $f$.
	There is a bijection between the set $\Phi_j \in P(\phi)$ such that
	$\Phi_j\neq \Phi_0$ and $\Phi_j$ fixes $u$, and the set of linear edges
	$E_j$ for $f$ such that $f(E_j) = E_ju^d$.
\end{lemma}

\begin{remark}
	The bijection depends on the choice of base lift but the collection
	$\{\Phi_0, \ldots, \Phi_{m(\mu)-1}\} \subset P(\phi)$ depends only on $u$
	and $\phi$. Given $\Phi \in P(\phi)$ that fixes $u$, this collection is
	equal to \[P(\phi) \cap \{i_u^n\circ \Phi\}_{n\in\mathbb{Z}}.\] 
\end{remark}

\begin{lemma}\label{lem:LKisaninvariant}
	The multisets $\mathcal{L}_{K}$, as $[K]$ varies over $\Fix(\phi)$,
  depend only on $\phi$ and not on the CT used to compute them.
\end{lemma}
\begin{proof}
We give an alternate definition of $\mathcal{L}_{K}$ in terms of the principal
	lifts and axes of $\phi$. (Compare Handel and Mosher's Fact
	II.2.8~\cite{HandelMosher}.) For each conjugacy class in $\Fix(\phi)$,
	fix a representative subgroup. We will define a multiset
	$\mathcal{L}'_K$ for each representative subgroup. 
	For each axis $\mu$ of $\phi$, pick a representative element $u$
	and an automorphism $\Phi \in P(\phi)$ fixing $u$. (As we are giving
	a CT independent definition, $\Phi$ may not be a base lift, but as
	remarked this will not matter.) The fixed subgroup $K_\Phi$ is
	conjugate by some $v$ to a representative $K$. For each $\Psi \in P(\phi)$ not equal
	to $\Phi$ that fixes $u$, the fixed subgroup $K_\Psi$ is conjugate by
	some $w$ to a
	representative $K'$ (it is possible $K = K'$ if $\Psi$ and $\Phi$ are
	isogredient), we add $i_v(u)$ to $\mathcal{L}_K$ and $i_w(u)$ to
	$\mathcal{L}_K'$. Following this procedure for all axes defines the
	multisets $\mathcal{L}'_K$, and up to conjugacy these are well-defined
	and depend only on $\phi$.

We claim that $\mathcal{L}_K = \mathcal{L}'_K$. 

Fix representatives of $\Fix(\phi)$. The axes of an automorphism $\phi$ are
	precisely the root-free conjugacy classes that are suffices of linear
	NEG edges. 
Let $\mu$ be an axis of $\phi$ with representative circuit $u$ and linear edges
$E_1,\ldots E_{m(\mu)-1}$ with suffix $u$. In Feighn and Handel's proof of
\Cref{axes-determine-linear-edges} a bijection is constructed as follows.
	Let $v$ be the initial (and terminal) vertex of $u$ and fix a lift
	$\tilde{v}\in \tilde{G}$. The base lift $\Phi_0$ is realized by the
	lift of $f$ at $\tilde{v}$. Each linear edge $E_j$ has a unique lift
	$\tilde{E}_j$ terminating at $\tilde{v}$, and the bijection sends $E_j$
	to the principal lift $\Phi_j$ of $f$ at the initial vertex of
	$\tilde{E}_j$. The fixed subgroup of $\Phi_0$ is conjugate via some $v$
	to the representative $K$; each $E_j$
	contributes a copy of $i_v(u)$ to $\mathcal{L}_{K}$ and to
	$\mathcal{L}'_{K}$ by definition.
	For each $j$ let $K_j$ be the fixed subgroup of
	$\Phi_j$, conjugate via some $w$ to its representative $K'$
	the edge $E_j$ contributes $i_w(E_ju\overline{E}_j)$ to
	$\mathcal{L}_{K}'$ by definition. As verified by Feighn and Handel this correspondence is a bijection~\cite{FH:RecogThm}*{Proof of Lemma
4.40}. Thus $\mathcal{L}_K = \mathcal{L}_K'$ is an invariant of $\phi$.
\end{proof}

It follows from the definition of a weak geometric model (in
particular the fact that $G_r$ embeds into $Y$) and the EG case of 
\Cref{moving-up-through-the-filtration} ``moving up through the filtration''
that each attaching map $\alpha_i$ of a lower boundary component is
either a local homeomorphism, or else has image $Ew\overline{E}$ for some
linear edge $E$ of $G$ and a closed loop $w$ which is a Nielsen path.
We now record the elements of $\F$ corresponding to boundaries of the surfaces
of EG strata in $G$.

\begin{definition}
	Define a multiset $\mathcal{E}_K$ of elements of $[K]\in \Fix(\phi)$ as follows
  ($\mathcal{E}$ is for exponential).  Let $H_r$ be a (necessarily
  geometric, c.f. \Cref{EGStrataAreGeometric}) EG stratum of $f\colon G\to
  G$, with associated weak geometric model as above. The upper boundary
  component is a closed
  height $r$ indivisible Nielsen path, which is represented in some component of
  $\hat{S}(f)$. Let $S$ be the surface associated to a weak geometric model
  $Y$ for $H_r$. The image under $\alpha_i$
  of each lower boundary component $\partial_iS$ is either a loop in
  some component of $\hat{S}(f)$, or an indivisible Nielsen path corresponding
  to a lollipop in a component $\hat{S}(f)$. The multiset $\mathcal{E}_K$ is
  the (multiset)-union over EG strata of $f$ of the boundary classes that are
	represented in the component of $\hat{S}(f)$ corresponding to $[K]$.
\end{definition}

The reader should note that the elements of $\mathcal{L}_K$ are always
root-free conjugacy classes; the definition asks us to take roots. On
the other hand this is not necessarily the case for elements of
$\mathcal{E}_K$. See \Cref{ex:SurfaceExample}.

\begin{definition}
	The \emph{candidate boundary multiset} of $[K]\in \Fix(\phi)$, denoted by
	\[\mathcal{C}_K = \mathcal{L}_K \dot\cup \mathcal{E}_K\] is the multiset
	union of the two multisets defined previously.
\end{definition}

As observed in their definitions, both multisets in the union are computable from a CT $f$ and the graph
$\hat{S}(f)$. Thus for each $[K] \in \Fix(\phi)$ (represented as a connected
component of $\hat{S}(f)$), the multiset $\mathcal{C}_K$ is computable.

\begin{lemma}\label{lem:geometricboundaryrealizable}
  If $\phi$ is a geometric rotationless outer automorphism
	then for every $[K]\in\Fix(\phi)$, $\mathcal{C}_K$ $\partial$-realizable in $K$.
\end{lemma}
\begin{proof}
  Since $\phi$ is geometric, there is a bijection between $\Fix(\phi)$ and
  the fixed subsurfaces of the Thurston normal form of a realization $g: 
	\Sigma\to \Sigma$. Fix $[K] \in \Fix(\phi)$ and the corresponding subsurface
  $\Sigma_K$. The components of $\partial \Sigma_K$ are partitioned into 3
  types: those that the Thurston normal form Dehn twists around, those that
  separate $\Sigma_K$ from pseudo-Anosov subsurfaces, and components of
  $\partial\Sigma$.

  \Cref{lem:LKisaninvariant} characterizes the elements of $\mathcal{L}_K$
  in terms of principal lifts; the surface and outer automorphism notions of
  principal lift coincide, so $\mathcal{L}_K$ is also determined by the Dehn
  twist boundary curves of $\Sigma_K$.

  It follows from \Cref{pASubsurfaces=EGStrata} and the definition of
  $\mathcal{E}_K$ that the components of $\partial \Sigma_K$ that separate 
  $\Sigma_K$ from a pseudo-Anosov subsurface are equal to $\mathcal{E}_K$.

  Thus, either $\Sigma_K$ is an annulus and $\mathcal{C}_K$ is a multiset with
  2 copies of the generator of $\pi_1\Sigma_K$, or $\mathcal{C}_K$ is
  a $\partial$-realizable set with $\Sigma_K$ witnessing the realization.
\end{proof}

As a consequence we further observe:

\begin{corollary}\label{cor:boundaryrealizablebehave}
	If $\phi$ is geometric, then for every $[K]\in\Fix(\phi)$,
  \begin{enumerate}
  \item $\mathcal{L}_K$ is a set -- the multiplicity of every element
    is $1$
  \item no element of $\mathcal{E}_K$ is a proper power.
  \end{enumerate}
\end{corollary}

\subsection{The algorithm}

\begin{algorithm}[h]
\caption{Decide if a rotationless outer automorphism is
	geometric.\label{alg:rotationless}}
\begin{algorithmic}[1]
\Procedure{RotationlessGeometric?}{$\phi$}
	\State $f \gets $ \Call{ComputeCT}{$\phi$}
	\If{\Call{ZeroStrata?}{$f$} $\vee$ \Call{NonlinearNEG?}{$f$} $\vee$
	\Call{NongeometricEG?}{$f$}}
		\State \Return \texttt{No} \label{strataconcern}
	\EndIf
	\State $\hat{S} \gets $ \Call{Compute$\hat{S}$}{$f$}
	\For{$K \in $ \Call{ConnectedComponents}{$\hat{S}$}}
		\State $\mathcal{C} \gets $ \Call{ComputeCandidateBoundary}{$K$, $f$}
		\If{$\neg$\Call{$\partial$-realizable?}{$\mathcal{C}$, $K$}}
			\State \Return \texttt{No} \label{realizableconcern}
		\EndIf
	\EndFor
	\State \Return \texttt{Yes}
\EndProcedure
\end{algorithmic}
\end{algorithm}

\begin{proposition}\label{rotationless-is-correct}
  \Cref{alg:rotationless} is correct.
\end{proposition}

\begin{proof}
  First \Cref{alg:rotationless} halts: every procedure used is an algorithm and there are finitely many
  components in $\hat{S}$ for any CT.

  Now suppose $\phi$ is a rotationless outer automorphism. There are two
  possible results of running \Cref{alg:rotationless} on $\phi$,
  we consider them in turn.

	\textsc{No}. If the algorithm returns \texttt{No} from \cref{strataconcern},
  then, by the contrapositive of \Cref{EGAndNEGStrataOff'} or
  \Cref{cor:nozeroornest}, $\phi$ is not
	geometric.  If the algorithm returns \texttt{No} from \cref{realizableconcern},
  then by the contrapositive of \Cref{lem:geometricboundaryrealizable}, $\phi$ is not geometric. In either
  case, it correctly reports non-geometric.

	\texttt{Yes}. Let $f$ be the CT for $\phi$ used by the algorithm. 
  Since the algorithm did not return from
  \cref{strataconcern}, we know that every EG stratum of $f$ is
  geometric, so let $\{S_i\}$ be the set of surface pieces of the weak
  geometric models for these strata, with pseudo-Anosov homeomorphisms $g_i$;
  The attaching maps determine an identification of each $S_i$ with a conjugacy
  class of subgroup of $\F$ and so a marking up to choice of basepoint.
  Further, 
  since the algorithm did not return from \cref{realizableconcern}, we
	know that for each $[K]\in \Fix(\phi)$ the
  multiset $\mathcal{C}_K$ is $\partial$-realizable in $K$; let $S_K$ be a
  surface witnessing this realization, marked by the 
  choice of representative $K$.

  Each element of $\mathcal{C}_K$ matches to a unique boundary component of a
  surface $S_K'$ or $S_i$, according to whether it came from $\mathcal{E}_K$ or
  $\mathcal{L}_K$. Since $\mathcal{C}_K$ is $\partial$-realizable, this
  identifies each boundary component of $S_K$ with a unique boundary component
  of some $S_K'$ or $S_i$.
  Moreover, if two EG pieces $S_i$ and
  $S_j$ have boundaries $c_i$ and $c_j$ which have conjugate proper powers, 
	then there is some $[K]\in \Fix(\phi)$ such that
  \[\{c_i, c_j\} \subset \mathcal{E}_K.\] It follows from
  $\partial$-realizability that $c_i$ and $c_j$ are not proper powers and
  indeed
  \[K = \langle c_i\rangle = \langle
  c_j^{\pm}\rangle.\]  So we can glue the surfaces $S_K$ and $S_i$ according to
  the boundary identifications coming from their boundary conjugacy classes in
  $\F$, in a way that extends the marking.
  Call the resulting marked surface $\Sigma$. 

  Define a homeomorphism $g:  \Sigma \to
  \Sigma$ by $g_i$ on each $S_i$ component, the identity on each $S_K$
  component, and a Dehn-twist by the twist power of $\phi$ around each curve
  from each $\mathcal{L}_K$. By construction $g_\ast$ and $\phi$ have the same
  set of laminations and the same twist coordinates, so by the
  Recognition Theorem~\cite{FH:RecogThm} for $\Out$ we conclude
  $g_\ast = \phi$, that is $\phi$ is geometric.
\end{proof}

\begin{porism}
	There is an algorithm \textsc{RotationlessGeometricWitness} that takes as input a geometric rotationless $\phi\in\Out$ and outputs a marked surface $\Sigma$ and distinguished subsurface $Q$ such that $\phi$ is realized by a homeomorphism $g$ on $\Sigma$, $g|_{\Sigma\setminus Q} = \mathrm{id}$ and $g$ is not isotopic to the identity on any subsurface of $Q$.
\end{porism}
\begin{proof}
	Observe that in the \texttt{Yes} case of the proof of \Cref{rotationless-is-correct} the proof describes how to construct $\Sigma$ from data computed from a CT for $\phi$, and that \[Q = \Sigma\setminus (\sqcup_{[K]\in\Fix(\phi)}S_K^\circ\] is also computable from data computed from a CT for $\phi$.
\end{proof}

\begin{remark}
	A consequence of \Cref{cor:godgivensurfacepiece,pASubsurfaces=EGStrata} is that the subsurface $Q$, carrying the pseudo-Anosov pieces of $\phi$ is determined by the outer automorphism $\phi$. This observation is essential for the root-finding algorithm.
\end{remark}

\section{Manipulating partially geometric surface pairs}\label{sec:partial}

Given an outer automorphism $\phi$ with rotationless power $\phi^N$, we can
use \Cref{alg:rotationless} to produce a marked surface pair $(\Sigma, Q)$
such that $\phi^N$ is geometric on $\Sigma$ realized by a homeomorphism $g$.
Moreover $g_{\Sigma\setminus Q^\circ} = \mathrm{id}$. (If $\phi$ is finite order
this produces an empty $Q$.) The subsurface $Q$ is an invariant of $\phi^N$,
and a necessary condition for $\phi$ to be geometric is that $\phi$ is
partially geometric on some other pair $(\Sigma', Q)$. Here we need to be
precise by what we mean by $(\Sigma', Q)$.

\begin{definition}
	Two marked surface pairs $(\Sigma, Q)$, $(\Sigma',Q')$ are
	\emph{outer-equivalent} if $Q$ is homeomorphic to $Q'$ and there is a
	difference-of-markings map $h\from \Sigma\to\Sigma'$ such that $h|_Q$
	is a homeomorphism and the induced outer automorphism $h_\ast$
	restricts to the identity on $\pi_1 Q$. For convenience we will abuse
	notation and often refer to outer-equivalent pairs $(\Sigma, Q)$ and
	$(\Sigma', Q)$.
\end{definition}

Observe that if $\phi$ is a partially geometric outer automorphism on $(\Sigma,
Q)$ then $\phi$ is also partially geometric on every outer-equivalent
$(\Sigma', Q)$. 

The primary way we will obtain outer-equivalent pairs is by deleting a
connected component $K\subset \Sigma\setminus Q^\circ$ and re-attaching a
different surface $K'$ along $\partial K$ such that $[\pi_1 K] = [\pi_1 K']$.
We refer to this operation as \emph{replacing a subsurface}, $K'$ is the
\emph{replacement} for $K$ and $\Sigma'$ is obtained from $\Sigma$ by \emph{replacing} $K$.

\begin{remark} \label{remark replacing a subsurface}
	If $\Sigma$ is triangulated and $K$ is a sub-triangulation replacing $K$ with a triangulated $K'$ is a computable operation.
\end{remark}

It is clear this operation will be useful in determining if the finite-order
behavior on the identity components of Thurston normal form for a rotationless
power are geometric. It turns out that this operation is also useful in
developing an algorithm to decide if an outer automorphism is partially
geometric, so we treat the notion in this section before continuing with the algorithm.

\begin{lemma} \label{lem:partially geometric with boundary}
	An outer automorphism $\phi$ is partially geometric on a pair
	$(\Sigma, Q)$ if and only if $\partial Q$ is $\phi$-invariant and there
	exists an outer-equivalent surface pair $(\Sigma',Q)$ and a subsurface
	$Q'\supseteq Q$ such that every connected component of $\partial Q'$
	has nontrivial intersection with $\partial \Sigma'$ and $\phi$ is partially geometric on $(\Sigma',
	Q')$. Moreover this extension is computable and $(\Sigma', Q')$ does not depend on
	$\phi$.
\end{lemma}

\begin{proof}
	First, we prove the forward direction. Since $\phi$ is partially
	geometric on $(\Sigma, Q)$, the boundary $\partial Q$ is
	$\phi$-invariant by definition.
	
	In each component of $\Sigma \setminus Q$ contains a boundary curve of
	$\Sigma$
	this is straightforward: for each component $c$ of $\partial Q$ that does
	not already meet $\partial \Sigma$ choose an arc connecting $c$ to
	$\partial \Sigma$, and chose this family of arcs to be disjoint. Then
	set $Q'$ to be the union of $Q$ and regular neighborhoods of these
	arcs. Since $Q'$ is isotopic to $Q$ we can modify a geometric witness
	for $\phi$ by a homotopy to obtain a geometric witness for $\phi$ on
	$(\Sigma, Q')$.

	If a component $R$ of $\Sigma \setminus Q$
	does not meet $\partial \Sigma$ and the fundamental group is not generated by the
	components of $\partial Q$ meeting $R$, replace each surface in the
	$\phi$-homotopy orbit of $R$ with a copy of a surface $R'$ with
	boundary and fundamental group of the same rank. By construction $\phi$
	induces a homotopy equivalence of $\Sigma'$, and we are now in the
	previously considered case. 
	If the component $R$ has
	fundamental group generated by the boundary, then by the
	Dehn-Nielsen-Baer theorem $\phi$ is homotopic to a homeomorphism on
	$R$, so adjust the realization of $\phi$ and take $Q' = Q\cup
	R$.

	Observe that all of these operations can be done with finite
	refinements of finite triangulations, so the extension is computable.

	Conversely, if such an extension exists then since $\partial Q$ is
	$\phi$-invariant $\phi$ has a realization as a homotopy equivalence of
	$\Sigma'$ that is a homeomorphism of $Q$, so $\phi$ is partially
	geometric on $(\Sigma', Q)$ and hence on $(\Sigma, Q)$.
\end{proof}

\section{Partial geometricity and the core}\label{sec:core}

In this section we give a criterion for an outer automorphism $\phi$ to be partially geometric on a surface pair $(\Sigma, Q)$ in terms of the Guirardel core. This
criterion is computable using Behrstock, Bestvina, and
Clay's~\cite{behrstockbestvinaclay} algorithm for computing certain cores, which leads to an algorithm for testing partial geometricity~(\Cref{partially-geometric-test}).
The starting point for this analysis is Guirardel's motivating observation that if $\gamma, \gamma'$ are a pair of filling curves on a closed surface $\Sigma$, then the core of the Bass-Serre trees coming from the induced splittings along $\gamma$ and $\gamma'$ is the universal cover of the square tiling dual to their minimalgeneral position intersection. Similar square tilings are central, leading to a definition.

\begin{definition}\label{def:relatively geometric}
	A nonempty connected marked square complex $m_X\from \mathfrak{R}\to X$ is \emph{surface type} if there is an embedding $\eta\from X\to \Sigma$ for a marked surface $(\Sigma,m_\Sigma)$ such that $m^{-1}_{\Sigma\ast}\eta_\ast m_{X\ast} = \mathrm{id}\in\Out$. A nonempty subcomplex $Y\subseteq X$ of a marked square complex is \emph{surface type} if each component of $Y$ is surface type with the induced marking.

	A free $\F$ action on a connected square complex $\tilde{X}$ is \emph{surface type} if it is the universal cover of a surface type marked square complex $X/\F$. A non-empty subcomplex $Y\subseteq X$ is \emph{surface type} if it is surface type with respect to the $\Stab(Y)$ action.
\end{definition}

\begin{remark}
	We allow square complexes that themselves are not homeomorphic to
	surfaces. A rose is a surface type square complex. Two squares joined at
	a single vertex is also a surface type square complex.
\end{remark}

\begin{definition}
	Suppose $X$ is a surface type marked square complex. A
	\emph{surface boundary class} of $X$ is a conjugacy class
	$[\gamma]\subset \F$ that can be represented by a connected component
	of the boundary of a regular neighborhood of an embedding
	$X$ in some marked surface $\Sigma$. The set of surface boundary
	classes for a particular embedding is denoted $\partial_\Sigma X$. This notion is similarly applied to complexes with free $\F$ action.
\end{definition}

Note that the definition of surface boundary class depends on the embedding, a
surface type complex may embed in more than one non-homeomorphic marked
surface.

To connect surface structures to square complexes, we also make use of a
standard combinatorial model of surfaces.

\begin{definition}
	Suppose $G \subseteq \Sigma$ is a graph embedded in a surface $\Sigma$.
	This embedding induces a \emph{ribbon structure} on $G$: an assignment
	of a rectangle to each edge, a polygon to each vertex, and gluing data
	attaching these rectangles to polygons such that the result is a cell
	complex homeomorphic to a regular neighborhood of $G$ in $\Sigma$.
	The \emph{dual arc system} to a ribbon structure is the collection of
	arcs transverse to $G$, each arc intersecting a unique edge exactly
	once and joining the unglued sides of the corresponding rectangle.
\end{definition}

The relationship between partial geometricity and the Guirardel core is explored in the
following two propositions. The
two propositions are presented separately, as the conclusion of
\Cref{prop:all-surface-all-spine} is significantly stronger than the hypothesis of
\Cref{prop:backward}, but together they imply a characterization of partial
geometricity in \Cref{corechar}.

\begin{proposition}\label{prop:all-surface-all-spine}
	Suppose $\phi$ is a partially geometric automorphism in $\Out$,
	realized on $(\Sigma, Q)$, such that every component of $\partial Q$
	meets $\partial \Sigma$. Let $G\subseteq \Sigma$ be a spine for
	$\Sigma$ with a subgraph $K$ that is a
	spine of $Q$. Let $T$ be the universal cover of $G$. For each pair of
	connected components $K_1, K_2$ of $K$ and each elevation $T_i$ of
	$K_i$ to $T$, the subcomplex $(T_1\times T_2\phi) \cap \Core(T,T\phi)$
	is surface type. Here $T_2\phi$ is the minimal subtree for $\phi(\Stab(T_2))$.
\end{proposition}

\begin{proof}

	Let $\mathcal{A}$ be the collection of arcs dual to the ribbon
	structure of $K$, representatives chosen so that the endpoints of each
	$\alpha\in \mathcal{A}$ are on $\partial \Sigma$. This is possible
	since each component of $\partial Q$ meets $\partial \Sigma$ and the
	dual arcs start and end on $\partial Q$.

	 Let $\Sigma_d$ be the 2-complex obtained by doubling $\Sigma$ along
	$(\Sigma\setminus Q) \cup \partial Q$. By Van Kampen's
	theorem\footnote{If $(\Sigma \setminus Q) \cup \partial Q$ has more than one component, this
	is a shorthand for the two-vertex, multi-edge graph of groups
	splitting of $\pi_1(\Sigma_d)$ induced by $(\Sigma \setminus Q) \cup
	\partial Q$}, the fundamental group is 
	\[\pi_1(\Sigma_d) =
	\pi_1(\Sigma)\ast_{\pi_1((\Sigma\setminus Q) \cup \partial Q)}\pi_1(\Sigma). \]
	We fix an
	identification of $\F$ with one of the two
	$\pi_1(\Sigma)$ factors of the amalgam, which will be used throughout
	the proof. The double of $\mathcal{A}$, denoted $\mathcal{A}_d$, is a
	a collection of embedded essential circles in $\Sigma_d$ that do not pass through
	any non-manifold points as no point of $\mathcal{A}$ is contained in
	$\Sigma \setminus Q$ by construction. Let $T$ be the universal cover of $G$, and let
	$A_d$ be the Bass-Serre tree of the splitting of $\pi_1(\Sigma_d)$
	defined by $\mathcal{A}_d$. By construction, the minimal subtree for
	the fixed $\F$ action, 
	$\bar{T} = A_d^{\F}$ of $A_d$ is the cover of the graph of groups obtained by
	collapsing the complement of $K$ in $G$. The tree $\bar{T}$ is a collapse of $T$.
	Indeed, for $\gamma\in\pi_1(\Sigma_d)$
	the translation length of $\gamma$ on $A_d$ is equal to the
	intersection number of $\gamma$ with $\mathcal{A}_d$, and for $\gamma\in \F$
	this is equal to the number of edges of $K$ in a cyclically reduced
	path in $G$ representing $\gamma$, again by construction.

	Let $g :\Sigma\to\Sigma$ be a partial geometric witness for $\phi$. Then
	$g(K)$ is also a
	spine for $Q$. There is a dual arc
	system $\mathcal{B} = g(\mathcal{A})$ dual to the induced ribbon graph
	structure on $g(K)$. Since $\partial Q$ meets $\partial \Sigma$ and $\partial Q$ is $g$-invariant we
	can take the endpoints of $\mathcal{B}$ to lie on $\partial \Sigma$.
	Moreover, since $g$ preserves the decomposition of
	$\Sigma$, it induces a map $\hat{g}\from \Sigma_d\to \Sigma_d$ on
	the double, and $\hat{g}(\mathcal{A}_d) =\mathcal{B}_d$. 
	As before, the double of the dual arc system
	$\mathcal{B}_d$ is a collection of embedded essential circles in
	$\Sigma_d$ that do not pass through any non-manifold points. Hence they
	induce a
	splitting of $\pi_1(\Sigma_d)$ and a Bass-Serre tree $B_d$. The
	minimal subtree $B_d^{\F}$ is the Bass-Serre tree for the
	graph of groups obtained by collapsing the complement of $K\phi$ in
	$G\phi$ (recall the action on the right is by twisting the marking, we are not using a topological representative of $\phi$ here). Observe that $B_d^{\F} = \bar{T}\phi$ and is a collapse of
	$T\phi$, and that we have the following commutative diagram of $\F$
	trees.
	\begin{center}
		\begin{tikzcd}
			T\ar{r}{\phi}\ar{d} & T\phi \ar{d}\\
			\bar{T} \ar{r}{\phi} &  \bar{T}\phi 
		\end{tikzcd}
	\end{center}

	Guirardel~\cite{guirardelcore}*{Section 2.2 Example 3} shows that the
	core of two Bass-Serre trees dual to curve systems on a surface is
	surface type and identifies the cell structure as the universal cover
	of the square complex dual to the intersection pattern, with connected
	components of the non-augmented core corresponding to the subsurfaces
	filled by the curve system. In fact, while Guirardel analyzes
	surfaces, the argument is local and applies to any collection of
	curves on a surface subset of a 2-complex. For clarity, we present the
	argument in full.

	Let $p \from \tilde{\Sigma}_d \to \Sigma_d$ be the universal cover of
	$\Sigma_d$. There are equivariant maps $f\from \tilde{\Sigma}_d \to
	A_d$ and $g\from \tilde{\Sigma}_d \to B_d$ since these trees are dual
	to the collections of curves. The components of $Q_d$ are closed
	surfaces with infinite fundamental group. Fix a constant
	curvature metric on each component of $Q_d$. This induces a metric on each component of $\tilde{Q}_d = p^{-1}(Q_d)$ making them either a hyperbolic or euclidean plane.

	Suppose
	$a\in A_d$ is an edge of $A_d$ dual to an element $\alpha \in
	\pi_1(\Sigma_d)$ and $b\in B_d$ is an edge of $B_d$ dual to $\beta \in
	\pi_1(\Sigma_d)$. Then, by construction, $\alpha$ and $\beta$ are represented by 
	closed curves contained in the subsurface $Q_d$. Let $\tilde{\alpha}$
	and $\tilde{\beta}$ be the axes of $\alpha$ and $\beta$ in
	$\tilde{\Sigma}_d$ (use a single representative parallel line in the case that $\alpha$ or $\beta$ represents a curve in a torus component of $Q_d$). 
	Analyzing the axes of other elements allows us to compute the partition
	for applying the 4-sets condition by analyzing connected components in
	the universal cover and applying the maps $f$ and $g$. There are four cases:

	\begin{description}
		\item[Case 1] $\tilde{\alpha}$ and $\tilde{\beta}$ are in
			distinct components of $\tilde{Q}_d$. Using the maps $f$ and $g$ it is clear that either
	$\partial^+(a)$ or $\partial^-(a)$ contains one of $\partial^\pm(b)$,
			so $a$ and $b$ do not satisfy the 4-sets condition of
			\Cref{Definition:Partition}. 
		\item[Case 2] $\tilde{\alpha}$ and $\tilde{\beta}$ are in
			a common hyperbolic plane component of  $\tilde{Q}_d$ but
			disjoint. Again, using the maps $f$ and $g$ it is
			clear that either $\partial^+(a)$ or $\partial^-(a)$
			contains one of $\partial^\pm(b)$, so $a$ and $b$ do
			not satisfy the 4-sets condition. 
		\item[Case 3] $\tilde{\alpha}$ and $\tilde{\beta}$ are in a
			common hyperbolic plane piece inside $\tilde{\Sigma}_d$ and
			intersect. Then, using the maps $f$ and $g$ we can see
			that each of the four sets
			\[ \partial^\pm(a)\cap \partial^\pm(b) \]
			have non-empty intersection, since there are curves
			which witness the non-empty
			intersection in this component of $\partial Q_d$.
		\item[Case 4] $\tilde{\alpha}$ and $\tilde{\beta}$ are in a
			common euclidean plane component of $\tilde{Q}_d$. Then
			either $\tilde{\alpha}$ is parallel to $\tilde{\beta}$,
			in which case, up to orientation
			$\partial^+(a)\cap\partial^-(b) = \emptyset$, or
			$\tilde{\alpha}$ and $\tilde{\beta}$ intersect, in
			which case $a$ and $b$ satisfy the 4-sets condition
			with witnesses coming from the plane as in the
			hyperbolic case.
	\end{description}

	Thus, the squares of $\Core(A_d, B_d)$ are in one-to-one correspondence with
	$\pi_1(\Sigma_d)$ orbits of intersection points of the curve systems
	$\mathcal{A}_d$ and $\mathcal{B}_d$. We conclude that the squares
	of the core 
	are the lift to the universal cover $\tilde{\Sigma}_d$ of the surface type square
	system dual to the intersection of $\mathcal{A}_d$ and $\mathcal{B}_d$.

	Further, restricting the $\pi_1(\Sigma_d)$ action
	to the $\F$ action coming from the fixed copy of $\pi_1(\Sigma)$ in the van Kampen decomposition, we obtain the equivariant inclusion 
	\[\Core(\bar{T}, \bar{T}\phi)\subseteq \Core(A_d, B_d).\]
	Since $A_d^{\F} = T$ and $B_d^{\F} = T\phi$, for each pair of components $K_1, K_2$ of $K$, and
	lifts $T_1, T_2$ to $T$,
	the intersection 
	\[ (T_1 \times T_2\phi) \cap \Core(T,T\phi) = (T_1 \times
	T_2\phi) \cap\Core(A_d^{\F},B_d^{\F})\] 
	is either empty or contained in a
	surface subset of $\tilde{\Sigma}_d$; i.e. the intersection is
	surface type, as required.
\end{proof}

\begin{proposition}\label{prop:backward}
	Suppose there is a marked graph $G$ with covering tree $T$ and a subgraph $K$ that carries a $\phi$-invariant free factor system for $\phi\in\Out$. Suppose
	\begin{itemize}
		\item the subcomplex $Y = \pi^{-1}(K) \subset \Core(T\times T\phi)/\F = X$ is surface type,
		\item there is an embedding such that the surface boundary classes of $Y$ are $\phi$-invariant,
		\item the graph of spaces decomposition of $X$ into $Y$ and $X\setminus Y^\circ$ is $\phi$-invariant.
	\end{itemize}
	Then $\phi$ is partially geometric on a marked surface pair $(\Sigma, Q)$.
\end{proposition}

\begin{proof}
	First, by hypothesis there is a homotopy equivalence $f\from X\to X$ representing $\phi$ such that $Y$ and $X\setminus Y^\circ$ are $f$-invariant.

	Let $Q$ be the regular neighborhood of an embedding of $Y$ into a surface $\Sigma$ such that $\partial Q$ is a set of $\phi$-invariant conjugacy classes. Such an object exists by hypothesis, and we may extend $f$ to $X\cup Q$. By the Dehn-Nielsen-Baer theorem (\Cref{dehn-nielsen-baer}), the restriction $f|_Q$ is homotopic to a map $g$ such that $g|_Q$ is a homeomorphism. Extending this homotopy by the constant homotopy on $X\setminus Q$ we obtain a homotopy equivalence $g\from X\to X$ such that $Q$ and $X\setminus Q$ are $g$-invariant and $Q|_g$ is a homeomorphism.

	Finally, the pair $(X,Q)$ is homotopy equivalent rel $Q$ to a surface pair $(\Sigma,Q)$. Since $Q$ and $X\setminus Q$ are $g$-invariant, $g$ induces a homotopy equivalence $h$ of $\Sigma$ such that $h|_Q = g|_Q$ is a homeomorphism, as required.
\end{proof}

\Cref{lem:partially geometric with boundary} allows us to combine
\Cref{prop:all-surface-all-spine} and
\Cref{prop:backward} to characterize partially geometric outer automorphisms.

\begin{theorem}\label{relcorechar}
	An outer automorphism $\phi\in\Out$ is partially geometric if and only if there
	exists a marked graph $G$ with cover $T$ such that $X = \Core(T, T\phi)/\F$
	has a surface type subcomplex $Y$ where $Y$ has $\phi$-invariant
	surface boundary classes and the splitting induced by $X\setminus Y$
	and $Y$ is $\phi$-invariant.
\end{theorem}

\begin{corollary}\label{corechar}
An automorphism in $\Out$
is geometric if and only if there exists some free simplicial $\F$-tree $T$
such that $\Core(T, T\phi)$ is geometric and has $\phi$-invariant surface
	boundary classes.
\end{corollary}

\begin{proof}[Proof of \Cref{relcorechar}]
	If $\phi$ is partially geometric on $(\Sigma, Q)$ then, by
	\Cref{lem:partially geometric with boundary} it is partially geometric
	on a surface pair $(\Sigma', Q')$ where every component of $\partial
	Q'$ intersects $\partial \Sigma'$. Let $G$ be a spine for $\Sigma'$ with
	a subgraph $K$ that is a spine for $Q'$ and covering tree $T$. By
	\Cref{prop:all-surface-all-spine} the core quotient $X = \Core(T,
	T\phi)/\F$ has a surface type subcomplex $Y = \pi^{-1}(K)$ that has
	$\phi$-invariant surface boundary classes and induces a
	$\phi$-invariant splitting.

	The converse is exactly the content of \Cref{prop:backward}.
\end{proof}

Finally, the criterion of \Cref{relcorechar} is computable. Moreover, it is
computable for a specific desired subgraph $K$, which allows us to determine if
an outer automorphism is partially geometric on a given marked surface pair.

\begin{corollary}\label{partially-geometric-test}
	Given an outer automorphism $\phi$, a marked surface $\Sigma$ and a
	subsurface $Q$ there is an algorithm \textsc{PartiallyGeometric?} to
	decide if $\phi$ is partially geometric on $(\Sigma,Q)$.
\end{corollary}

\begin{proof}
	It is algorithmic to verify if $\partial Q$ is $\phi$-invariant, and
	report \texttt{No} if not. Next
	use \Cref{lem:partially geometric with boundary} to compute an
	outer-equivalent pair $(\Sigma', Q)$ and an extension $Q'\supset Q$
	such that every component of $\partial Q'$ meets $\partial \Sigma'$.
	By \Cref{lem:partially geometric with boundary} $\phi$ is partially
	geometric on $(\Sigma, Q)$ if and only if it is partially geometric on
	$(\Sigma', Q')$. By \Cref{relcorechar} and \Cref{prop:all-surface-all-spine}, $\phi$ is partially geometric on $(\Sigma', Q')$ if and
	only if for any spine $G$ of $\Sigma'$ with subgraph $K$ carrying $Q'$,
	the subcomplex $\pi^{-1}(K)$ of the core quotient $
	\Core(T,T\phi)/\F$ is surface type with $\phi$ invariant surface
	boundary and splitting. 

	It is algorithmic to compute a desired spine $G$ and subgraph $K$.
	Moreover the core quotient $\Core(T,T\phi)/\F$ where $T$ is the cover
	of $G$ is finite can be computed using
	Behrstock, Bestvina, and Clay's algorithm~\cite{behrstockbestvinaclay}. Finally, it is
	algorithmic to verify if a finite subcomplex is surface type, there are
	finitely many possible surface boundary classes, and it is algorithmic
	to test $\phi$ invariance. 
\end{proof}

\Cref{partially-geometric-test} alone is not sufficient to determine if an
outer automorphism is geometric algorithmically, as it does not offer any
method for searching the space of marked surfaces in finite time. However, in
combination with \Cref{alg:rotationless} it can be used to reduce the problem
of determining if an outer automorphism is geometric to an extension problem,
which we explore next.

\section{From partially geometric to geometric}\label{sec:partial-to-geometric}

With a method to test if an outer automorphism is partially geometric in hand,
we turn our attention to deciding an extension problem: given a partially geometric outer
automorphism $\phi$ realized by $g$ on the marked pair $(\Sigma, Q)$ such that 
\[ g^N|_{\Sigma\setminus Q^\circ} \sim \mathrm{id},\] 
can it be extended to be
partially geometric on a larger subsurface, partially after replacing
components of $\Sigma\setminus Q$.
 We will solve the extension problem at hand by giving an
algorithm to replace surfaces (or certify that no geometric replacements
exist).

Our first extension result deals with the special case of finding a geometric replacement for a $\phi$-invariant splitting factor.

\begin{lemma} \label{finite-order-on-fixed-set-extension}
	Suppose $\phi\in\Out$ is partially geometric on $(\Sigma, Q)$. Let $R
	\subset \Sigma \setminus Q^\circ$ be a connected component of the
	complement such that $\phi([\pi_1 R]) = [\pi_1 R]$. There is an
	algorithm to produce a finite list $R_1, \ldots, R_s$ of replacements for $R$ such that 
	\begin{itemize}
		\item $\phi$ is partially geometric on $(\Sigma_i, R_i \cup Q)$, where $\Sigma_i$ is the result of replacing $R$ with $R_i$;
		\item all possible replacements $R'$ such that $\phi$ is partially geometric on $(\Sigma', R' \cup Q)$ are listed, up to the action of the $\out(\pi_1 R)$ stabilizer of  $[\partial Q]\cap [\partial R]$.
	\end{itemize}
\end{lemma}

\begin{proof}
	Let $\mathcal{C} = [\partial Q \cap \partial R]$ be the set of
	conjugacy classes joining $R$ to $Q$. Let $n = \mathrm{rank}(\pi_1R)$ and
	$m\from \mathfrak{R}_n\to R$ be the induced geometric
	marking from the inclusion $\pi_1R\to \F$. 
	Observe that $\mathcal{C}$ is $\partial$-realizable, and that
	\emph{any} $\partial$-realization of $\mathcal{C}$ can be used as a replacement for $R$.

	Suppose there is a replacement $R'$ of $R$ such that $\phi$ is partially
	geometric on $(\Sigma', R'\cup Q)$, the result of replacing $R$ with
	$R'$ in $\Sigma$. Then $\phi|_R$ in the $\out(R)$
	conjugacy class of a finite order mapping class for of the
	standard surface $S_{n,i}$ homeomorphic to $R$. Moreover, the
	conjugating element $\theta$ satisfies $\theta(\mathcal{C}) \subseteq
	[\partial S_{n,i}]$.

	Conversely, suppose $\phi|_R$ is conjugate by $\theta$ in $\out(R)$ to some
	$\alpha$ which is a finite order mapping class on a standard surface
	$S_{n,i}$, and $\theta(\mathcal{C})\subseteq[\partial S_{n,i}]$.  Then
	$(S_{n,i}, \theta)$ is a marked replacement $R'$ for $R$ such that
	$\phi|_{R'}$ is geometric, realized by a homeomorphism representative of $\alpha$. Thus $\phi$ is partially
	geometric on $(\Sigma', R'\cup Q)$ as desired.

	This condition allows us to use the equivariant Whitehead algorithm of Krstic, Lustig, and Vogtmann~(\Cref{klv-whitehead}). The restriction $\alpha =
	\phi|_R$ is finite order. For each standard surface $S_{n,i}$,
	enumerate representatives of the finitely many $\Mod(S_{n,i})$
	conjugacy classes of finite order mapping classes as $\alpha_{n,i,j}$.
	For each $\alpha_{n,i,j}$ use the equivariant Whitehead algorithm to
	decide if there is some $\theta\in \out(R)$ such that
	$\theta\alpha\theta^{-1} = \alpha_{n,i,j}$ and
	$\theta(\mathcal{C})\subset [\partial S_{n,1}]$. Use each result to
	report a replacement $R'$; if none are found report the empty list.
\end{proof}

As a consequence, we obtain an algorithm to handle the general finite-order case. 

\begin{corollary} \label{finite-order-algorithm}
	There is an algorithm \textsc{FiniteOrderGeometric?} to decide if a finite order outer automorphism is geometric.
\end{corollary}

\begin{proof}
	Use the above algorithm with $Q = \emptyset$.
\end{proof}

 Note that, if $\phi$ is partially geometric on $(\Sigma, Q)$, a witness for $\phi$ homotopy permutes the complementary components of $Q$. So, it remains to handle the replacement of entire $\phi$ orbits.

\begin{lemma} \label{geometric-orbit-replacement}
	Suppose $\phi\in \Out$ is partially geometric on
	$(\Sigma, Q)$, realized by $g\from \Sigma\to \Sigma$. Let $R_0 \subset \Sigma\setminus Q^\circ$ be a connected
	component of the complement such that $[\pi_1 R_0]$ is $\phi$-periodic with period $k$ and $\phi^k|_{\pi_1 R_0}$ is finite-order. Let $R_0,\ldots, R_{k-1}$ be the components visited by the forward orbit of $R_0$ under $g$.

	There is an algorithm \textsc{GeometricOrbitReplacement?} that takes $\phi$, $\Sigma$, and $R_0$ and decides if there are replacements $R_0',\ldots, R_{k-1}'$ such that $\phi$ is partially geometric on $(\Sigma', Q\cup R')$ where $R' = \bigcup_{i=0}^{k-1} R_i$. Moreover there is a procedure to produce $\Sigma'$ if it exists.
\end{lemma}

\begin{proof}
	First, apply \Cref{finite-order-on-fixed-set-extension} to $\phi^k$
	and $R_0$, to obtain a list of candidate replacements $\{R_{0,i}\}$.
	If this list is empty, report \texttt{No}; if $\phi$ is geometric then
	$\phi^k$ will have at least one geometric replacement for $R_0$, and finitely many up to the action of $\out(\pi_1 R_0)$. 

 Observe that if there is a family of replacements as
	in the hypotheses, $\phi([\partial R_0'])$ is a $\partial$-realizable
	subset of $[\pi_1 R_1]$. Since $\partial Q$ is $\phi$-invariant, a $\partial$-realization of $\phi([\partial R_0'])$ can replace $R_1$.

	Moreover, $R_0'$ is homeomorphic to a candidate replacement $R_{0,i}$.
	Let $m', m_i$  be the respective markings so that $m = \theta m'$
	for some $\theta\in\out(\pi_1 R_0)$, then 
	\[ \phi m = \phi \theta m' = \phi \theta \phi^{-1} \phi m'. \]
	Thus $\phi m'$ and $\phi m_i$ differ by the action of $\out(\pi_1
	R_1)$. We deduce that there is a family of replacements as in the
	hypothesis if, and only if, for some candidate replacement $R_{0,i}$,
	the set $\phi^j([\partial R_{0,i}])$ is $\partial$-realizable in $\pi_1 R_j$ for all $0 < j < k$.

	To finish the algorithm, for each candidate replacement $R_{0,i}$, check

	\[ \text{\textsc{$\partial$-realizable?}}(\phi^j([\partial R_{0,i}]), \pi_1
	R_j),\]
	for each $0 < j < k$.
	If each iterate is $\partial$-realizable,
	report \texttt{Yes} and provide the $\partial$-realizations of the
	iterates as the desired replacements. Otherwise, report \texttt{No}. \Cref{remark replacing a subsurface} assures us that $\Sigma'$ can be computed if desired.
\end{proof}

\section{Deciding geometricity in general}\label{finalalgo} 

\begin{algorithm}[h]
\caption{Decide if an outer automorphism is
	geometric.\label{alg:final}}
\begin{algorithmic}[1]
\Procedure{Geometric?}{$\phi$}
	\State $\psi \gets$ \Call{RotationlessPower}{$\phi$}
	\If{$\psi = \mathrm{id}$}
		\State \Return \Call{FiniteOrderGeometric?}{$\phi$}
	\EndIf
	\If{$\neg$\Call{RotationlessGeometric?}{$\psi$}}
		\State \Return \texttt{No} \label{rotationlessageometric}
	\EndIf
	\State $(\Sigma_0, Q_0) \gets$ \Call{RotationlessGeometricWitness}{$\psi$}
	\State $(\Sigma, Q) \gets$
	\Call{BoundaryContactExtension}{$\Sigma_0$,$Q_0$}
	\If{$\neg$\Call{PartiallyGeometric?}{$\phi$,$\Sigma$,$Q$}}
		\State \Return \texttt{No} \label{notroot}
	\EndIf
	\For{$R \in $ \Call{ConnectedComponents}{$\Sigma\setminus Q^\circ$}}
		\If{$\neg$\Call{GeometricOrbitReplacement?}{$\phi$,$\Sigma$,$R$}}
			\State\Return \texttt{No} \label{badorbit}
		\EndIf
	\EndFor
	\State \Return \texttt{Yes}
\EndProcedure
\end{algorithmic}
\end{algorithm}

\begin{theorem} \label{main-theorem}
	\Cref{alg:final} is correct.
\end{theorem}

\begin{proof}
	First, \Cref{alg:final} halts: there are finitely many connected
	components of $\Sigma\setminus Q^\circ$ as $\Sigma, Q$ are compact
	surfaces.

	Next, \Cref{alg:final} reports the correct result. If $\phi$ is
	geometric, then any power is geometric; thus if the \texttt{No} result
	comes from \cref{rotationlessageometric} it is correct. Next, since $Q$
	is the subsurface where the rotationless power $\phi^N$ is
	non-identity, it is determined up to topological type and marking by
	\Cref{cor:godgivensurfacepiece} and \Cref{pASubsurfaces=EGStrata}, so a necessary condition for $\phi$ to be
	geometric is that $\phi$ is relatively geometric relative to $(\Sigma,
	Q)$; thus \cref{notroot} is correct. By
	\Cref{geometric-orbit-replacement}, if
	$\phi$ is geometric for each homotopy orbit of connected complementary component
	there is a geometric replacement. So the algorithm correctly reports
	\texttt{No} at \cref{badorbit} if any orbit fails to have a geometric
	replacement.

	Finally, if $\phi$ is relatively geometric on $(\Sigma, Q)$ and every
	homotopy orbit of connected complementary component has a geometric
	replacement, then $\phi$ is geometric on $\Sigma'$, the result of
	conducting all of these geometric replacements, so the report of \texttt{Yes} is correct.
\end{proof}

\begin{porism}
	There is an algorithm \textsc{GeometricWitness} that takes as input an outer automorphism $\phi\in\Out$ and produces a marked surface $\Sigma$ where $\phi$ is realized by a homeomorphism.
\end{porism}

\begin{proof}
	If $\phi$ is geometric, the surface $\Sigma'$ constructed in the proof of \Cref{main-theorem} is the desired marked surface. The process of finding and attaching geometric replacements is algorithmic (\Cref{geometric-orbit-replacement}).
\end{proof}

\section{Illustrative examples} \label{sec:example}

\begin{example}\label{ex:SurfaceExample}
  Consider the genus two surface with two punctures
  $\Sigma=\Sigma_{2,2}$ depicted in \Cref{fig:SurfaceExample}
  thought of as a union of a pair of pants $\Sigma'=\Sigma_{1,0}^2$
  and a torus with one puncture and two boundary components
  $\Sigma''=\Sigma_{1,1}^2$, glued appropriately.  Let
  $h\colon \Sigma''\to\Sigma''$ be a pseudo-Anosov homeomorphism and
  let $g\colon\Sigma\to\Sigma$ be defined by
  $g=h\circ D_\beta^4\circ D_\gamma^2$ where $D_c$ is the right Dehn
  twist about the curve $c$.

  There is a nested sequence of $g$-invariant subsurfaces (starting
  with a neighborhood of $\alpha$)
  $N_\alpha\subset N_\alpha\sqcup N_\beta\subset \Sigma'\subset
  \Sigma$ that, on the level of fundamental groups, determines a
  nested sequence of $g_*$-invariant free factor systems of
  $\pi_1\Sigma=\F_5$.  A CT $f\colon G\to G$ realizing this nested
  sequence will have six strata (see the figure for a schematic):
  $H_i$ for $i=1,2,3$ each consist of a single fixed edge $E_i$, with
  the first two being disjoint loops (corresponding to $\alpha$ and
  $\beta$ respectively) and the third connecting these loops; $H_4$ is
  a linear edge, $E_4$, with $f(E_4)=E_4E_2^4$; $H_5$ is a linear edge
  $E_5$ with $f(E_5)=E_5(E_2\overline{E_3}E_1E_3)^2$; finally $H_6$ is
  a geometric EG stratum whose geometric model has two lower boundary
  components with attaching maps
  $\alpha_1(\partial_1 S)=E_4E_2\overline{E_4}$ and
  $\alpha_2(\partial_2S)=E_5E_2\overline{E_3}E_1E_3\overline{E_5}$.

  \begin{figure}[h]
    \centering{ \def\svgwidth{.9\linewidth}
      \begingroup \makeatletter \providecommand\color[2][]{\errmessage{(Inkscape) Color is used for the text in Inkscape, but the package 'color.sty' is not loaded}\renewcommand\color[2][]{}}\providecommand\transparent[1]{\errmessage{(Inkscape) Transparency is used (non-zero) for the text in Inkscape, but the package 'transparent.sty' is not loaded}\renewcommand\transparent[1]{}}\providecommand\rotatebox[2]{#2}\newcommand*\fsize{\dimexpr\f@size pt\relax}\newcommand*\lineheight[1]{\fontsize{\fsize}{#1\fsize}\selectfont}\ifx\svgwidth\undefined \setlength{\unitlength}{937.5bp}\ifx\svgscale\undefined \relax \else \setlength{\unitlength}{\unitlength * \real{\svgscale}}\fi \else \setlength{\unitlength}{\svgwidth}\fi \global\let\svgwidth\undefined \global\let\svgscale\undefined \makeatother \begin{picture}(1,0.6)\lineheight{1}\setlength\tabcolsep{0pt}\put(0,0){\includegraphics[width=\unitlength,page=1]{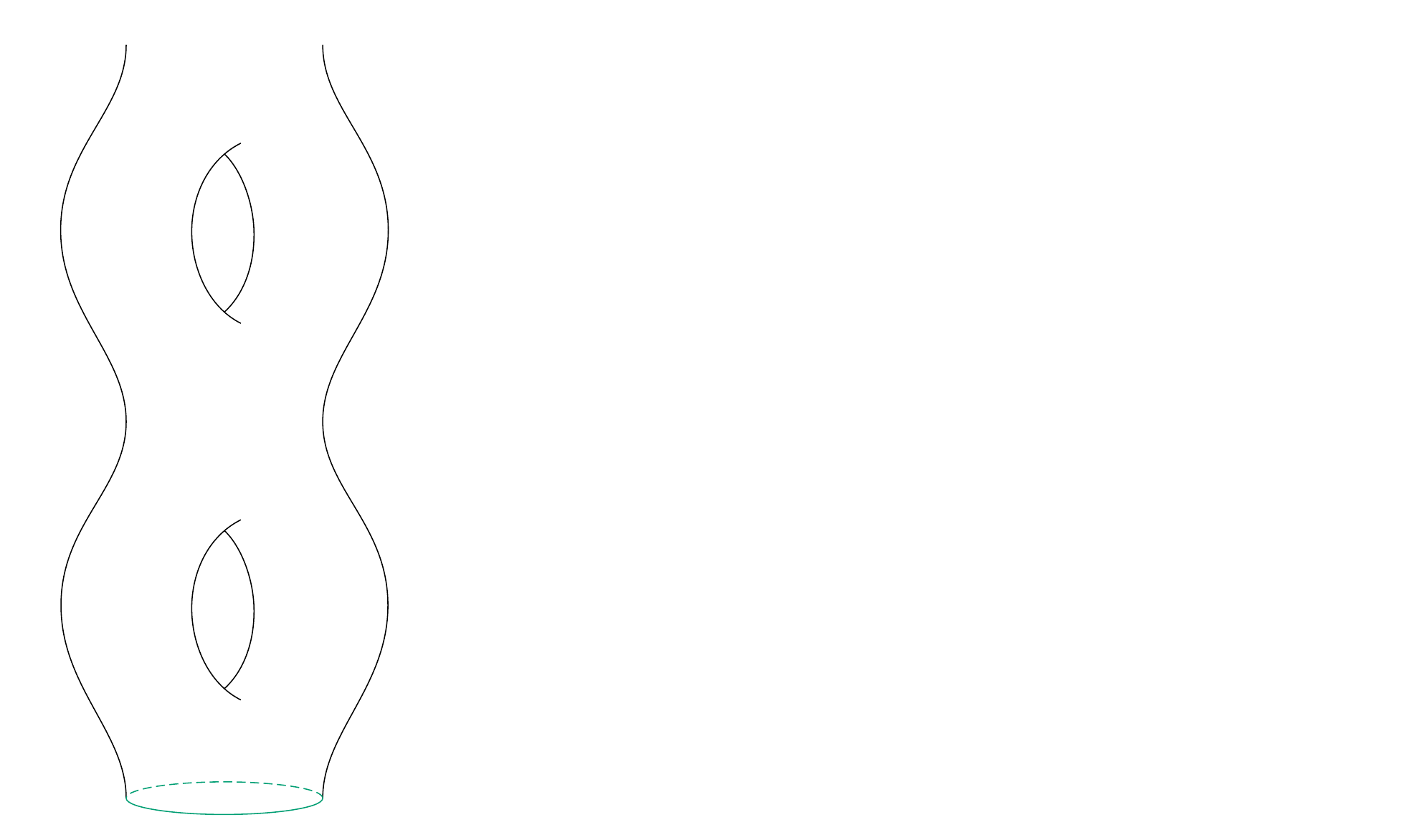}}\put(0.248,0.02){\color[rgb]{0,0.61960784,0.45098039}\makebox(0,0)[lt]{\lineheight{1.25}\smash{\begin{tabular}[t]{l}$\alpha$\end{tabular}}}}\put(0,0){\includegraphics[width=\unitlength,page=2]{SurfaceExample.pdf}}\put(0.244,0.564){\color[rgb]{0,0,0}\makebox(0,0)[lt]{\lineheight{1.25}\smash{\begin{tabular}[t]{l}$\delta$\end{tabular}}}}\put(0.012,0.5424){\color[rgb]{0,0,0}\makebox(0,0)[lt]{\lineheight{1.25}\smash{\begin{tabular}[t]{l}$\Sigma$\end{tabular}}}}\put(0.00080001,0.28){\color[rgb]{0,0,0}\makebox(0,0)[lt]{\lineheight{1.25}\smash{\begin{tabular}[t]{l}$g\colon \Sigma\curvearrowright$\end{tabular}}}}\put(0.00720001,0.1568){\color[rgb]{0.83529412,0.36862745,0}\makebox(0,0)[lt]{\lineheight{1.25}\smash{\begin{tabular}[t]{l}$\beta$\end{tabular}}}}\put(0,0){\includegraphics[width=\unitlength,page=3]{SurfaceExample.pdf}}\put(0.28151458,0.15745073){\color[rgb]{0,0,0}\makebox(0,0)[lt]{\lineheight{1.25}\smash{\begin{tabular}[t]{l}$\gamma$\end{tabular}}}}\put(0,0){\includegraphics[width=\unitlength,page=4]{SurfaceExample.pdf}}\put(0.7008,0.0656){\color[rgb]{0,0,0}\makebox(0,0)[lt]{\lineheight{1.25}\smash{\begin{tabular}[t]{l}$E_1$\end{tabular}}}}\put(0,0){\includegraphics[width=\unitlength,page=5]{SurfaceExample.pdf}}\put(0.9008,0.0656){\color[rgb]{0,0,0}\makebox(0,0)[lt]{\lineheight{1.25}\smash{\begin{tabular}[t]{l}$E_2$\end{tabular}}}}\put(0,0){\includegraphics[width=\unitlength,page=6]{SurfaceExample.pdf}}\put(0.81560762,0.05238823){\color[rgb]{0,0,0}\makebox(0,0)[lt]{\lineheight{1.25}\smash{\begin{tabular}[t]{l}$E_3$\end{tabular}}}}\put(0,0){\includegraphics[width=\unitlength,page=7]{SurfaceExample.pdf}}\put(0.78319995,0.4592){\color[rgb]{0,0,0}\makebox(0,0)[lt]{\lineheight{1.25}\smash{\begin{tabular}[t]{l}$\delta$\end{tabular}}}}\put(0,0){\includegraphics[width=\unitlength,page=8]{SurfaceExample.pdf}}\put(0.89776869,0.1786){\color[rgb]{0,0,0}\makebox(0,0)[lt]{\lineheight{1.25}\smash{\begin{tabular}[t]{l}$\gamma$\end{tabular}}}}\put(0,0){\includegraphics[width=\unitlength,page=9]{SurfaceExample.pdf}}\put(0.94231614,0.3028){\color[rgb]{0,0,0}\makebox(0,0)[lt]{\lineheight{1.25}\smash{\begin{tabular}[t]{l}$E_4$\end{tabular}}}}\put(0,0){\includegraphics[width=\unitlength,page=10]{SurfaceExample.pdf}}\put(0.82000004,0.56){\color[rgb]{0,0,0}\makebox(0,0)[lt]{\lineheight{1.25}\smash{\begin{tabular}[t]{l}$\Fix(\phi)$\end{tabular}}}}\put(0,0){\includegraphics[width=\unitlength,page=11]{SurfaceExample.pdf}}\put(0.7376,0.1784){\color[rgb]{0,0,0}\makebox(0,0)[lt]{\lineheight{1.25}\smash{\begin{tabular}[t]{l}$\beta$\end{tabular}}}}\put(0,0){\includegraphics[width=\unitlength,page=12]{SurfaceExample.pdf}}\put(0.78054744,0.3018){\color[rgb]{0,0,0}\makebox(0,0)[lt]{\lineheight{1.25}\smash{\begin{tabular}[t]{l}$E_5$\end{tabular}}}}\put(0,0){\includegraphics[width=\unitlength,page=13]{SurfaceExample.pdf}}\put(0.2942,0.3374){\color[rgb]{0,0,0}\makebox(0,0)[lt]{\lineheight{1.25}\smash{\begin{tabular}[t]{l}$f\colon G\to G$\end{tabular}}}}\put(0.3696,0.5496){\color[rgb]{0,0,0}\makebox(0,0)[lt]{\lineheight{1.25}\smash{\begin{tabular}[t]{l}$G$\end{tabular}}}}\put(0,0){\includegraphics[width=\unitlength,page=14]{SurfaceExample.pdf}}\put(0.38879995,0.1904){\color[rgb]{0,0,0}\makebox(0,0)[lt]{\lineheight{1.25}\smash{\begin{tabular}[t]{l}$E_1$\end{tabular}}}}\put(0,0){\includegraphics[width=\unitlength,page=15]{SurfaceExample.pdf}}\put(0.58879995,0.1904){\color[rgb]{0,0,0}\makebox(0,0)[lt]{\lineheight{1.25}\smash{\begin{tabular}[t]{l}$E_2$\end{tabular}}}}\put(0,0){\includegraphics[width=\unitlength,page=16]{SurfaceExample.pdf}}\put(0.50360756,0.17718823){\color[rgb]{0,0,0}\makebox(0,0)[lt]{\lineheight{1.25}\smash{\begin{tabular}[t]{l}$E_3$\end{tabular}}}}\put(0,0){\includegraphics[width=\unitlength,page=17]{SurfaceExample.pdf}}\put(0.60992465,0.3999334){\color[rgb]{0,0,0}\makebox(0,0)[lt]{\lineheight{1.25}\smash{\begin{tabular}[t]{l}$H_6$\end{tabular}}}}\put(0,0){\includegraphics[width=\unitlength,page=18]{SurfaceExample.pdf}}\put(0.44879995,0.4824){\color[rgb]{0,0,0}\makebox(0,0)[lt]{\lineheight{1.25}\smash{\begin{tabular}[t]{l}$\delta$\end{tabular}}}}\put(0,0){\includegraphics[width=\unitlength,page=19]{SurfaceExample.pdf}}\put(0.47279995,0.2744){\color[rgb]{0,0,0}\makebox(0,0)[lt]{\lineheight{1.25}\smash{\begin{tabular}[t]{l}$E_4$\end{tabular}}}}\put(0,0){\includegraphics[width=\unitlength,page=20]{SurfaceExample.pdf}}\put(0.58059995,0.2746){\color[rgb]{0,0,0}\makebox(0,0)[lt]{\lineheight{1.25}\smash{\begin{tabular}[t]{l}$E_5$\end{tabular}}}}\put(0,0){\includegraphics[width=\unitlength,page=21]{SurfaceExample.pdf}}\end{picture}\endgroup        \caption{A surface homeomorphism $g\colon \Sigma\to\Sigma$ with
        associated CT $f\colon G\to G$.}
      \label{fig:SurfaceExample}
    }
  \end{figure}
  
  The graph $\hat{S}(f)$ is shown on the right side of the figure.  It has four components, each of whose fundamental groups is a subgroup of $\F_5$ that is preserved by a particular $\Phi\in[\phi]$; we denote these subgroups by $K_3,\ldots, K_6$ according to the stratum of the CT.  We now compute the sets $\mathcal{L}_{K_i}$ and $\mathcal{E}_{K_i}$ in this example; rather than write out the edge paths, we will abuse notation and write the corresponding curve on the surface.  Indeed,
  \begin{equation*}
    \mathcal{L}_{K_3}=\{\beta,\gamma\}\qquad
    \mathcal{L}_{K_4}=\{E_4\beta\overline{E_4}\}\qquad
    \mathcal{L}_{K_5}=\{E_5\gamma\overline{E_5}\}\qquad
    \mathcal{L}_{K_6}=\varnothing.
  \end{equation*}
 And
    \begin{equation*}
    \mathcal{E}_{K_3}=\varnothing\qquad
    \mathcal{E}_{K_4}=\{E_4\beta\overline{E_4}\}\qquad
    \mathcal{E}_{K_5}=\{E_5\gamma\overline{E_5}\}\qquad
    \mathcal{E}_{K_6}=\{\delta\}.
  \end{equation*}
\end{example}

The reader may find it instructive to consider how these sets change
when $\Sigma$ is replaced with: the 2-complex obtained by ``making a
duplicate copy of $\Sigma$'', or the 2-complex obtained by cutting
$\Sigma$ along $\beta$ and then re-gluing the resulting upper boundary
$\beta^+$ to the lower boundary $\beta^-$ via with a 2-1 map of the
circle.

\begin{example}
Consider the following reducible automorphism:

\begin{align*}
\phi:  & a \rightarrow ab^{5} \\
  & b \rightarrow b \\
\end{align*}

Let $T$ be the universal cover of the rose whose loops are labeled $a$ and $b$. 
The quotient of the core $\Core(T, T\phi)/\F$ is shown in \Cref{coreoftwist},
	this was computed with~\cite{train-track}; the red edge paths of length
	3 are identifed.  The projection from this core to the horizontal
	direction is a homotopy equivalence (the edge labelled $a$ is the
	diagonal of a rectangle) which provides a marking. 
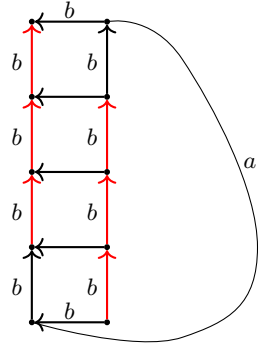
\begin{figure}[H]
\begin{tikzpicture}
 \tikzstyle{vertex} =[circle,draw,fill=black,thick, inner sep=0pt,minimum size=.5 mm]

   \node[vertex] (a1) at (0, 1) [label={[xshift=-0.2cm, yshift = 0.2cm] \small  $b$}]  {};
   \node[vertex] (a2) at (0, 2) [label={[xshift=-0.2cm, yshift = 0.2cm]  \small  $b$}]  {};
   \node[vertex] (a3) at (0, 3)  [label={[xshift=-0.2cm, yshift = 0.2cm] \small  $b$}]  {};
   \node[vertex] (a4) at (0, 4)  [label={[xshift=-0.2cm, yshift = 0.2cm]  \small $b$}]  {};
   \node[vertex] (a5) at (0, 5) {};
    \node[vertex] (b1) at (1, 1)  [label={[xshift=-0.2cm, yshift = 0.2cm] \small  $b$}]  {};
   \node[vertex] (b2) at (1, 2)  [label={[xshift=-0.2cm, yshift = 0.2cm] \small  $b$}]  {};
      \node[vertex] (b3) at (1, 3)  [label={[xshift=-0.2cm, yshift = 0.2cm] \small  $b$}]  {};
   \node[vertex] (b4) at (1, 4) [label={[xshift=-0.2cm, yshift = 0.2cm] \small  $b$}]  {};
   \node[vertex] (b5) at (1, 5){};

    \draw [->, thick] (a1)--(a2);
     \draw [->, thick, red] (a2)--(a3);
      \draw [->, thick, red] (a3)--(a4);
       \draw [->, thick, red] (a4)--(a5);
       
    \draw [->, thick] (b4)--(b5);
       \draw  [->, thick, red]  (b1)--(b2);
        \draw  [->, thick, red]  (b2)--(b3);
         \draw  [->, thick, red]  (b3)--(b4);
         
    \draw [->, thick] (b1)--(a1)  {};
    \node at (a1) [label={[xshift=0.5cm, yshift = -0.2cm] \small  $b$}]  {};
     \draw [->,thick] (b2)--(a2);
      \draw[->,thick]  (b3)--(a3);
       \draw[->,thick]  (b4)--(a4);
        \draw [->, thick] (b5)--(a5);
         \node at (a5) [label={[xshift=0.5cm, yshift = -0.2cm] \small  $b$}]  {};
          \node at (2.4, 3) [label={[xshift=0.5cm, yshift = -0.2cm] \small  $a$}]  {};
 
 \pgfsetplottension{0.75}
  \pgfplothandlercurveto
  \pgfplotstreamstart
  \pgfplotstreampoint{\pgfpoint{0cm}{1cm}}  
  \pgfplotstreampoint{\pgfpoint{2cm}{0.8cm}}   
    \pgfplotstreampoint{\pgfpoint{3cm}{2cm}} 
  \pgfplotstreampoint{\pgfpoint{2cm}{4.5cm}}
  \pgfplotstreampoint{\pgfpoint{1cm}{5cm}}
  \pgfplotstreamend
 \pgfusepath{stroke}

\end{tikzpicture}
\caption{$\Core(T, T\phi)/\F$ for a typical Dehn twist.\label{coreoftwist}}
\end{figure}
The core is shown in \Cref{cup} and can be embedded in a surface, the edge
labels indicate the horizontal marking (roughly) to guide the embedding.
\begin{figure}[H]
\begin{tikzpicture}[scale=0.5]

 \draw[thick]  (0,0) ellipse (2 and 1);
   \draw [thick] (0,5) ellipse (2 and 1);
   
  \node at (0, 1.5) [label={[xshift=0cm, yshift = -0.4cm] \small  $b$}]  {};
  
   \node at (0, 6.5) [label={[xshift=0cm, yshift = -0.4cm] \small  $b$}]  {};
  
    \draw [thick] (2, 5) to (2, 0);
    \draw [->>, thick, red] (-2, 0) to (-2, 5);
    \draw(2.6,5.35) to[bend left=15](4, 4.1);
     \node at (2, 3) [label={[xshift=0.5cm, yshift = -0.2cm] \small  $a$}]  {};
      \node at (3.7, 1.5) [label={[xshift=0.5cm, yshift = -0.2cm] \small  $a$}]  {};
  
 \pgfsetplottension{0.80}
  \pgfplothandlercurveto
  \pgfplotstreamstart
  \pgfplotstreampoint{\pgfpoint{2cm}{5cm}}  
  \pgfplotstreampoint{\pgfpoint{4cm}{4cm}}   
    \pgfplotstreampoint{\pgfpoint{4cm}{1cm}} 
  \pgfplotstreampoint{\pgfpoint{2cm}{0cm}}
  \pgfplotstreamend
 \pgfusepath{stroke}
       
        \pgfsetplottension{0.75}
  \pgfplothandlercurveto
  \pgfplotstreamstart
  \pgfplotstreampoint{\pgfpoint{1cm}{5.9cm}}  
  \pgfplotstreampoint{\pgfpoint{3cm}{4.9cm}}   
    \pgfplotstreampoint{\pgfpoint{3cm}{1.9cm}} 
  \pgfplotstreampoint{\pgfpoint{1cm}{0.9cm}}
  \pgfplotstreamend
 \pgfusepath{stroke}

\end{tikzpicture}
\caption{The surface $S$ in which the Guirardel core embeds.\label{cup}}
\end{figure}
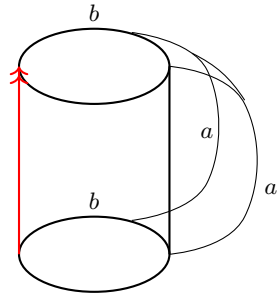
There is only one boundary component in this surface, which is associated via
the marking to the word $ab^{-1}a^{-1}b$ (half of the top and bottom $b$ . The surface is homeomorphic to a torus with one boundary component. 
One can check that:
\[\phi(ab^{-1}a^{-1}b) = ab^{-1}a^{-1}b.\]

\end{example}

\bibliographystyle{alpha}

\end{document}